\documentclass[12pt]{amsart}

\usepackage{amscd,amssymb,latexsym,amstext,amsmath,amsthm,amsfonts,xypic}
\usepackage{graphicx}
\usepackage{epsfig}
\usepackage{color}
\usepackage{palatino, url, multicol}
\usepackage[all,cmtip]{xy}

\newtheorem{theo}{Theorem}
\newtheorem{definition}{Definition}
\newtheorem{lemma}{Lemma}

\theoremstyle{definition}
\newtheorem{case}{Case}

\title[Infinite genus surfaces and  irrational polygonal billiards]
{Infinite genus surfaces and irrational polygonal billiards}
\author[Ferr\'an Valdez]{Ferr\'an Valdez}
\address{Max Planck Institut f\"{u}r Mathematik
Vivatsgasse 7.
53111, Bonn, Germany.}
\email{ferran@mpim-bonn.mpg.de}

\begin{document}
		

\begin{abstract}
We prove that the natural invariant surface associated with the billiard game on an irrational polygonal table is homeomorphic to the Loch Ness monster, that is, the only orientable infinite genus topological real surface with exactly one end.
\end{abstract}
\maketitle


\section{introduction}
The classical Katok-Zemljakov construction associates to every rational angled polygon $\rm P$ a compact flat surface $\rm  X(P)$ with finite angle conic singularities \cite{KZ}, \cite{Tr}. The genus of this surface is totally determined by the angles of the polygon:
\begin{lemma}\cite{G}
\label{l0}
Let the interior angles of $\rm  P$ be $\rm \pi m_i/n_i$, $i=1,\ldots,k$, where $\rm m_i$ and $\rm n_i$ are coprime, and let $\rm N$ be the least common multiple of the $\rm n_i$'s. Then
$$
\rm
\text{genus}\hspace{1mm}  X(P)=1+\frac{N}{2}\left(k-2-\Sigma\frac{1}{n_i}\right)
$$
\end{lemma}
The Katok-Zemljakov construction, in the case of a general nondegenerate simply connected polygon $\rm P$, leads to a non-compact flat surface $\rm  X(P)$. Our main result is to determine its topological type: 
\begin{theo}
	\label{teo}
Let $\rm\lambda_1\pi,\ldots,\lambda_N\pi$ be the interior angles of the polygon $\rm P$. Suppose that there exists $\rm j=1,\ldots,N$ such that  $\rm \lambda_j$ is not a rational number. Then, the flat surface $\rm  X(P)$ is homeomorphic to the Loch Ness monster, that is, the only infinite genus, orientable topological surface with exactly one end.
\end{theo}
\indent Theorem 1 extends some results presented in \cite{V} that deal with the case where $\rm P$ is a "generic" triangle. Moreover, the proof we present makes no use of the  dictionary between flat surfaces arising from general polygons and the leaves of homogeneous holomorphic foliations on $\mathbf{C}^2$.\\
\indent This article is organized as follows. In Section \ref{gen}  we recall some basic concepts and results related to non-compact orientable surfaces and flat surfaces arising from polygons. In section \ref{original} we prove Theorem \ref{teo}. \\

\indent \textbf{Acknowledgments}. The author gratefully acknowledges support from Sonderforschungsbereich/Transregio 45 and ANR Symplexe and thanks the Max-Planck Institut f\"{ur} Mathematik in Bonn for the excellent working conditions provided during the preparation of this paper. The author is indebted to M. M\"{o}ller, B. Weiss for useful comments and discussions, and would like to express his gratitude to the referee for many valuable remarks.
\section{generalities}\label{gen} From now on, \emph{surface} will mean a connected 2-dimensional real \emph{orientable} manifold.
\subsection{Non-compact orientable surfaces: the Loch Ness monster} A \emph{subsurface} of a given  surface is a closed region inside the surface whose boundary consists of a finite number of non-intersecting simple closed curves. The genus $\rm g$ of a \emph{compact} bordered surface $\rm S$  with $\rm q$ boundary curves and Euler characteristic $\chi$ is the number $\rm g=1-\frac{1}{2}(\chi+q)$. A surface is said to be \emph{planar} if all of its compact subsurfaces are of genus zero.
\begin{definition}
A surface $\rm X$ is said to have \emph{infinite genus} if there exists no finite set $\rm \mathcal{C}$ of mutually non-intersecting simple closed curves with the property that $\rm X\setminus \mathcal{C}$ is  \emph{connected and planar}.
\end{definition}
\indent Two compact real surfaces are homeomorphic if and only if they have the same genus. Ker\'ekj\'art\'o's theorem states that non-compact surfaces of the same genus and orientability class are homeomorphic if and only their \emph{ideal boundaries} are homeomorphic. We refer the reader to \cite{R} for the definition of ideal boundary and a proof of this theorem. Points in the ideal boundary of a surface $\rm X$ are also called \emph{ends} of $\rm X$.
\begin{definition}
Up to homeomorphism, the \emph{Loch Ness monster} is the unique infinite genus surface with only one end.
\end{definition}
\begin{center}
\includegraphics[scale=0.2]{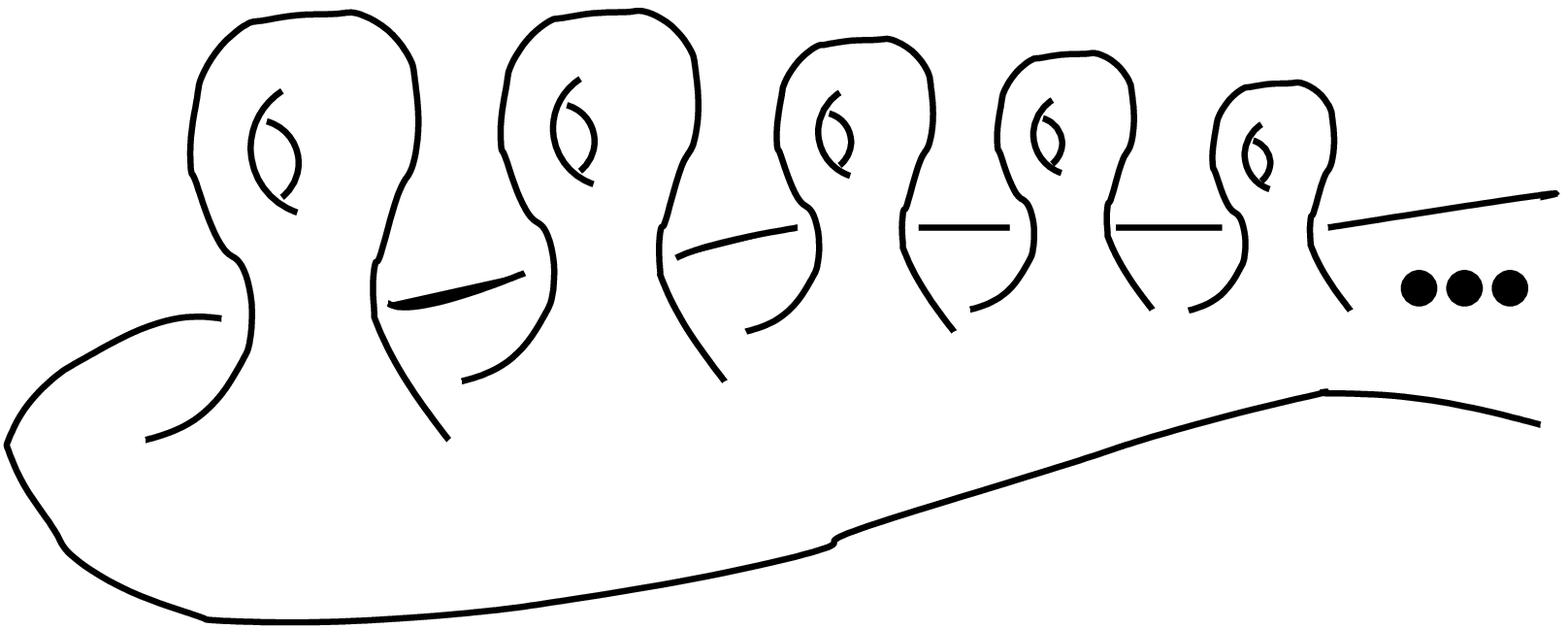}\\
Figure 1: the Loch Ness monster.\\
\end{center}
\indent This nomenclature can be found in \cite{Gh}. An infinite genus surface $\rm X$ has only one \emph{end} if and only if, for every compact set $\rm K\subset X$, there exists a compact set $\rm K\subset K'$ such that $\rm X\setminus K'$ is connected \cite{R}.\\
\\
\subsection{Flat coverings of the N-punctured sphere}\label{flatco} Henceforth $\rm P\subset\mathbf{R}^2$ denotes a non-degenerate \emph{simply connected} N-sided polygon. Vertices of $\rm P$ at which the interior angle is a rational multiple of $\rm \pi$ are called \emph{rational}. All vertices of a rational polygon are rational.\\
\indent Let $\rm P_0:=P\setminus\{\text{vertices of } P\}$. The identification of two disjoint copies of the vertexless polygon $\rm P_0$ along "common sides" defines a Euclidean structure on the $\rm N$-punctured sphere. We denote it by $\rm\mathbf{S}^2(P)$. The following figure illustrates the simplest case:
\begin{center}
  \label{nends}
\includegraphics[scale=0.25]{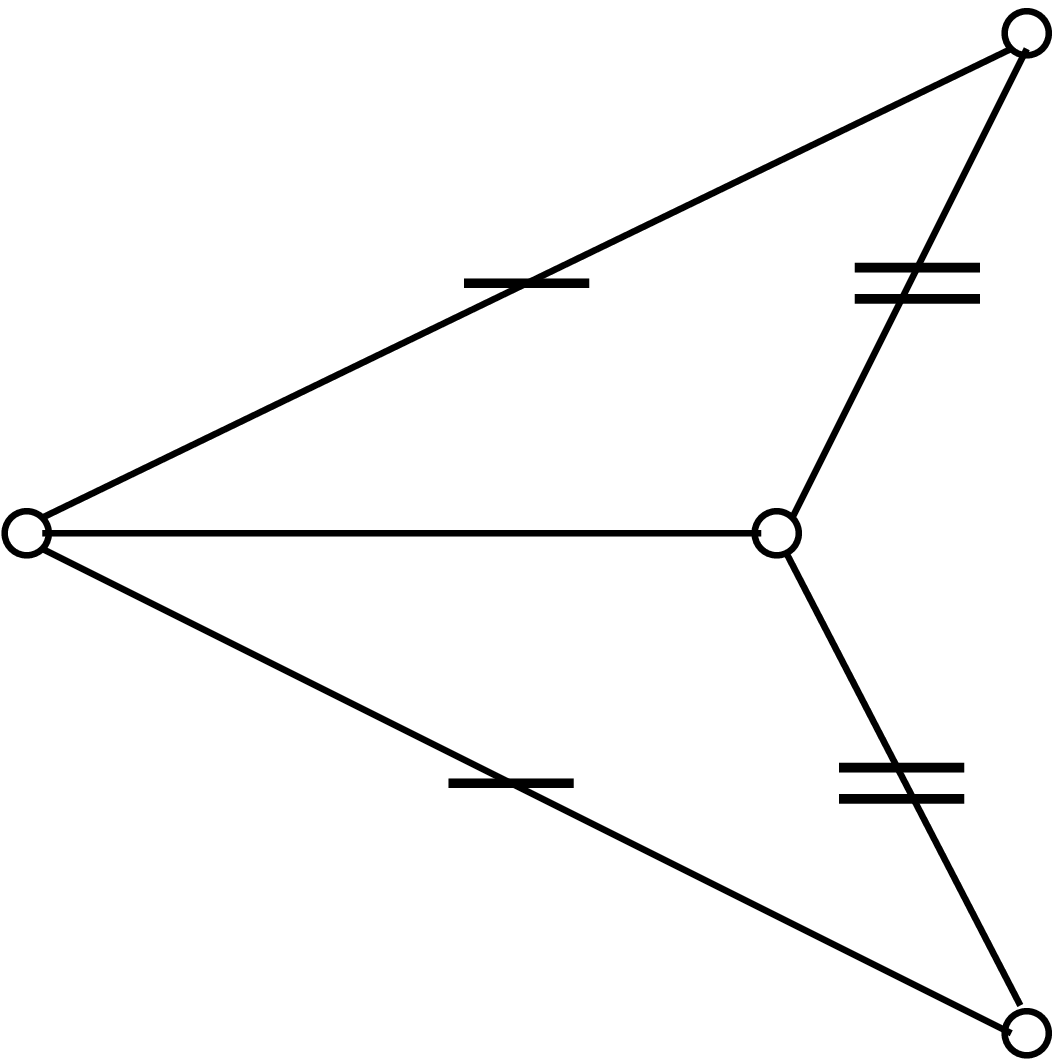}\\
Figure 2: $\rm \mathbf{S}^2(P)$ for a triangle $\rm P$.\\
\end{center}
\indent The locally Euclidean structure on $\rm\mathbf{S}^2(P)$ gives rise to the holonomy representation:
\begin{equation}
  \label{hol}
\rm hol:\pi_1(\mathbf{S}^2(P))\longrightarrow Isom_+(\mathbf{R}^2)
\end{equation}
\indent  Let $\rm B_j $ be a simple loop in $\rm \mathbf{S}^2(P)$ around the missing vertex of $\rm P$ whose interior angle is $\rm\lambda_j\pi$, $\rm j=1,\ldots, N$. Suppose that $\rm B_j\cap B_i=*$, for $\rm i\neq j$. Then, $\rm \{B_j\}_{j=1}^N$ generates $\rm\pi_1(\mathbf{S}^2(P),*)$. Given an isometry $\rm\varphi\in Isom_+(\mathbf{R}^2)$, we denote its derivative by $\rm D\circ\varphi$. A simple calculation shows that $ \rm M_j:= D\circ hol(B_j)$ is given by:
  \begin{equation}
  \label{rots}
  \rm
  M_j=
\begin{pmatrix}
  \cos(2\lambda_j\pi) & -\sin(2\lambda_j\pi) \\
  \sin(2\lambda_j\pi) & \cos(2\lambda_j\pi)
\end{pmatrix}
\rm
\hspace{1cm} j=1,\ldots,N.
 \end{equation}
\indent Let $ \rm\widetilde{\mathbf{S}^2(P)}$ be the universal cover of the N-punctured sphere $\rm \mathbf{S}^2(P)$ and $\rm Trans(P)$ the kernel of $\rm D\circ hol$.
\begin{definition} \cite{Ho}
Let $\rm X(P_0):=\widetilde{\mathbf{S}^2(P)}/Trans(P)$. The translation surface $\rm  X(P_0)$ is called the \emph{minimal translation surface corresponding to $\rm P$}.
\end{definition}
\indent  We denote by $\rm \hat{\pi}:X(P_0)\longrightarrow \mathbf{S}^2 (P)$ the corresponding projection. The deck transformation group of the covering, $\rm Aut(\hat{\pi})$, is abelian, as $\rm Trans(P)$ always contains the commutator subgroup\linebreak $\rm [\pi_1(\mathbf{S}^2(P)),\pi_1(\mathbf{S}^2(P))]$ (see [\emph{ibid}], \cite{Ho1}).\\
\indent When the set of rational vertices of $\rm P$ is not empty, the translation surface $\rm X(P_0)$ can be locally compactified by adding to it points above rational vertices of $\rm P$. The result of this local compactification is a flat surface with conical singularities. We denote it by $\rm X(P)$. If the set of rational vertices of $\rm P$ is empty, we set $\rm X(P)=X(P_0)$. 
\begin{definition}
We call $\rm X(P)$ the flat surface obtained from the polygon $\rm P$ by the Katok-Zemljakov  construction.
\end{definition}
 \emph{Remark}. In the case of rational polygons, some authors give a different definition for the flat surface $\rm X(P)$, see \cite{MT} or \cite{Tr}.
 
\section{Proof of theorem 1} 
	\label{original}
\indent The proof of Theorem \ref{teo} is organized as follows.  First we show that the surface $\rm X(P)$ has only one end. For this we consider two main cases: (A) all interior angles of $\rm P$ are irrational multiples of $\pi$ and (B) there exists at least one interior angle of $\rm P$ which is a rational multiple of $\pi$ . Second, we prove that $\rm X(P)$ has infinite genus. Through the proof of Theorem 1, we need the following
\begin{definition}
	\label{resonance}
	Let $\rm (n_1,\ldots,n_{N-1})$ be a choice of coordinates for $\rm \mathbf{Z}^{N-1}$ and $\rm\lambda'=(\lambda_{i_1},\ldots,\lambda_{i_{N-1}})$ a choice of $\rm N-1$ angle parameters $\rm\lambda_{i_j}\in\{\lambda_1,\ldots,\lambda_{N}\}$. A point $\rm \hat{n}:=(n_1,\ldots,n_{N-1})$ is called a resonance of $\lambda'$ if and only if $\rm\sum_{j=1}^{N-1}n_j\lambda_{i_j}$ is an integer. The set of resonances of $\lambda'$ is a subgroup of $\rm\mathbf{Z}^{N-1}$. Let $\rm Res(\lambda')$ denote this subgroup and $\rm G(\lambda'):=\mathbf{Z}^{N-1}/Res(\lambda')$.  The parameter $\lambda$ is called \emph{totally irrational} if, for every choice $\lambda'$, the subgroup $\rm Res(\lambda')$ is trivial.
 \end{definition}
 \textbf{The surface $\rm X(P)$ has only one end.  Case (A)}. We suppose that all interior angles of the polygon $\rm P$ are irrational multiples of $\pi$, so that $\rm X(P_0)=X(P)$. Our strategy is to construct, for every compact set $\rm K\subset  X(P)$, a compact set $\rm K'=K'(K)$ containing $\rm K$ such that $\rm X(P)\setminus K'$ is connected.\\
\indent Let $\rm d$ be the standard Euclidean metric on $\mathbf{C}$ and 
\begin{equation}
	\label{rho}
	\rm
	\rho:\mathbf{S}^2(P)\longrightarrow P
\end{equation}
the natural projection of the punctured sphere onto the polygon $\rm P$. Denote the set of vertices of $\rm P$ by $\rm Ver(P)$.  For every compact set  $\rm K\subset X(P)$, the projection $\rm \hat{\pi}(K)$ onto the punctured sphere $\rm\mathbf{S}^2(P)$ is contained in:
\begin{equation}
	\label{kepsilon}
	\rm
	K_{\epsilon}:=\{t\in\mathbf{S}^2(P)\hspace{1mm}\mid\hspace{1mm} d(\rho(t),Ver(P))\geq\epsilon\},
\end{equation}
for a suitable choice of $0<\epsilon\ll1$. Let $\rm Sides(P)\subset P$ be the sides of the polygon and fix a side $\rm s\in Sides(P) $. We define
\begin{equation}
	\label{uepsilon}
	\rm
	U_{\epsilon}:= K_{\epsilon}\cap\rho^{-1}(P\setminus \{Sides(P)\setminus s\})
\end{equation}
\indent The closure of $\rm U_{\epsilon}$ in $\rm \mathbf{S}^2(P)$ is $\rm K_{\epsilon}$. For every $\rm\xi \in\hat{\pi}^{-1}(U_{\epsilon})$, we denote by $\rm\widetilde{U_{\epsilon}(\xi)}$ a connected lift of $\rm U_{\epsilon}$ to the surface $\rm X(P)$ containing the point $\xi$. By compactness, there exist finitely many points $\rm \{\xi_1,\ldots,\xi_m\}$ in $\rm X(P)$ such that the compact set $\rm K$ is contained in the closure in $\rm X(P)$ of
\begin{equation}
	\label{kprima}
	\rm
	\bigcup_{j=1 }^m\widetilde{U_{\epsilon}(\xi_j)}.
\end{equation}
\indent We define $\rm K'(K)$ to be the closure in $\rm X(P)$ of (\ref{kprima}). 
\begin{center}
\includegraphics[scale=0.3]{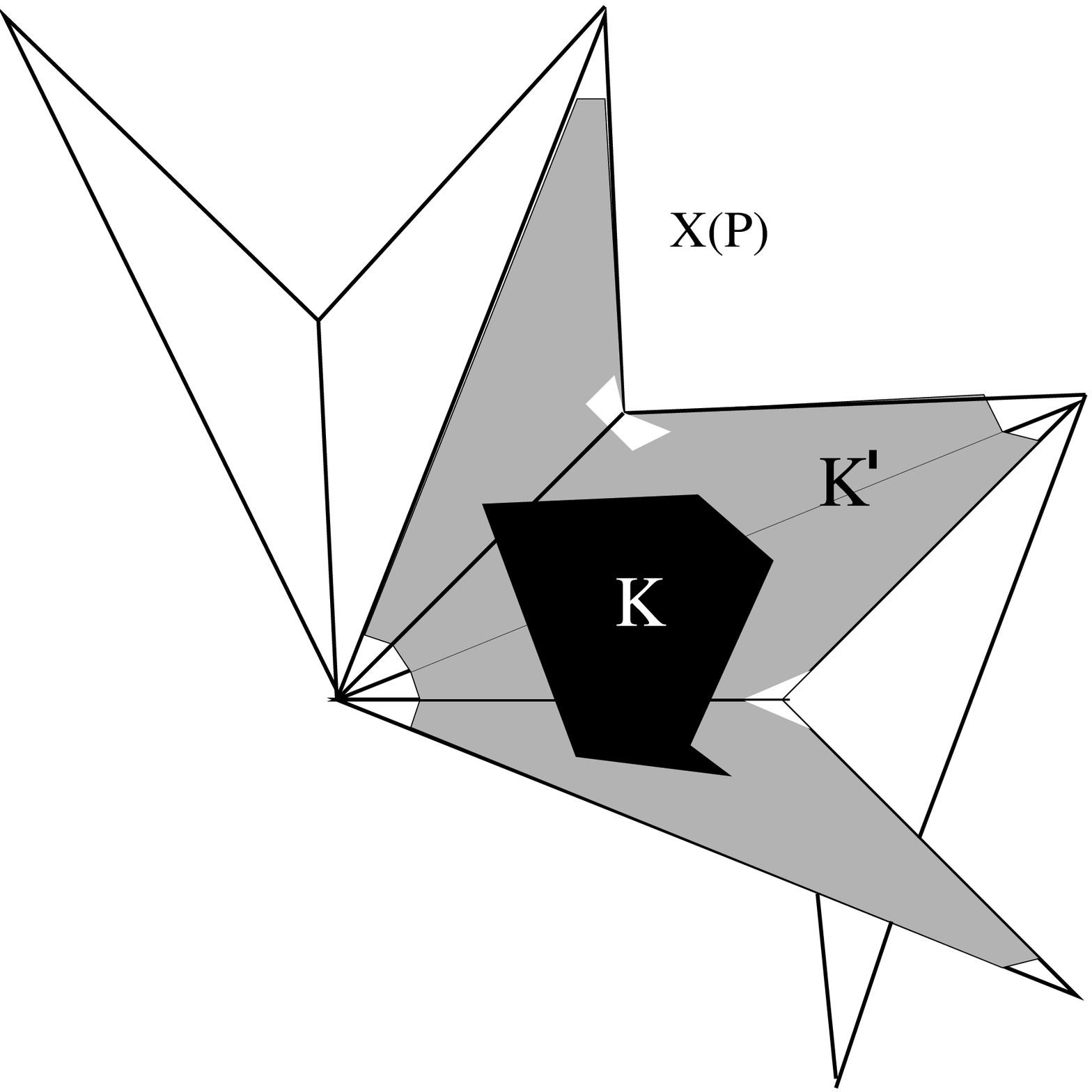}\\
Figure 3: The sets $\rm K$ and $\rm K'$ in $\rm X(P)$.
\end{center}
\begin{lemma}
The complement of $\rm K'$ in $\rm X(P)$ is connected by arcs. 
\end{lemma}
\begin{proof}
\indent Let $\rm\Gamma:=\{B_{i_j}\}_{j=1}^{N-1}$, $\rm i_j\in\{1,\ldots,N\}$, be a basis for $\rm \pi_1(\mathbf{S}^2(P),*)$ such that $\rm D\circ hol (B_{i_j})$ is a rotation of angle $\rm 2\pi\lambda_{i_j}$, for $\rm j=1,\ldots N-1$. For this choice of basis, let $\rm\lambda'=(\lambda_{i_1},\ldots,\lambda_{i_{N-1}})$.
The $\rm N$ punctured sphere $\rm\mathbf{S}^2(P)$ retracts to $\Gamma$.
Fix one retraction
\begin{equation}
	\label{R}
	 \rm R:\mathbf{S}^2(P)\twoheadrightarrow\Gamma. 
\end{equation}
It induces a retraction
\begin{equation}
	\label{Rtilde}
	 \rm \widetilde{R}: X(P)\twoheadrightarrow \Upsilon. 
\end{equation}
\indent In the following paragraph we describe the retraction space $\Upsilon$ in terms of $\rm C(\lambda')$, the Cayley graph of $\rm G(\lambda')$. First, we identify the set of vertices of $\rm C(\lambda')$, that by definition is the set of elements of the group $\rm G(\lambda')$, with the set of points $\rm \widetilde{R}(\hat{\pi}^{-1}(*))$, in the following way. Let $\rm\{e_j\}_{j=1}^{N-1}$ be the standard basis vectors for $\rm G(\lambda')$.  Recall that $\rm\widetilde{R}$ restricted to the fiber $\hat{\pi}^{-1}(*)$ is the identity. Choose a point $\rm *_0\in\hat{\pi}^{-1}(*)$ and define $\rm \mu: \hat{\pi}^{-1}(*)\longrightarrow G(\lambda')$ by 
\begin{equation}
	\label{mu}
	\rm
\mu(*_0)=(0,\ldots,0), \hspace{5mm}\mu(\rm B_{i_j}*_0)=e_j,\hspace{3mm}\text{for all}\hspace{1mm}j=1,\ldots,N-1.
\end{equation}
and, for every word $\rm m=m(B_{i_1},\ldots, B_{i_{N-1}})\in\pi_1(\mathbf{S}^2(P),*)$
\begin{equation}
	\label{mu1}
	\rm
\mu(m*_0)=m(\mu(B_{i_1}*_0),\ldots,\mu(B_{i_{N-1}}*_0))\mu(*_0).
\end{equation}
On the left side of the preceding equations, multiplication corresponds to the monodromy action of  $\rm \pi_1(\mathbf{S}^2(P),*)$ on the fiber $\hat{\pi}^{-1}(*)$, and on the right side, multiplication corresponds to the translation action of $\rm G(\lambda')$ on itself. Now we describe the identification for the edges of $\rm C(\lambda')$.  Let $\rm g$, $\rm e_jg\in G(\lambda')$ be vertices of $\rm C(\lambda')$. Denote $\rm\widetilde{B_{i_j}}(g)$ the lift of $\rm B_{i_j}$ via $\hat{\pi}$ whose extremities are the points $\rm \mu^{-1}(g)$ and $\rm \mu^{-1}(e_jg)$. We identify the edge between the vertices $\rm g$, $\rm e_jg$ with $\rm\widetilde{R}(\widetilde{B_{i_j}}(g))=\widetilde{B_{i_j}}(g)$. Henceforth $\Upsilon$ is endowed with its standard graph topology. This graph depends on our choice of basis $\Gamma$. Nevertheless, the number of ends of $\Upsilon$ does not, for $\rm\sum_{j=1}^{N}\lambda_j$ is an integer and this implies that $\rm Rank(G(\lambda'))=Rank(G(\lambda''))$ for any pair of choices $\lambda'$, $\lambda''$ of angle parameters. We are now in position to finish the proof of Lemma 2 by considering two cases.\\
\\
\underline{Case A.1: Rank $\rm G(\lambda')\geq 2$}. In this situation, any basis $\Gamma$ for \linebreak $\rm \pi_1(\mathbf{S}^2(P),*)$ defines a graph $\Upsilon$ with exactly one end. Therefore, if we set  $\rm \mathcal{K}:=\widetilde{R}(K')$, there is a finite subgraph  $\mathcal{K}'$ of $\Upsilon$ containing $\mathcal{K}$  such that $\rm\Upsilon\setminus\mathcal{K}'$ is connected. Let $\eta_1$ and $\eta_2$ be two points in $\rm X(P)\setminus K'$. We claim that $\eta_1$ and $\eta_2$ can be connected through an arc with points $\rm \eta'_1$ and $\eta'_2$ in $\rm\widetilde{R}^{-1}(\Upsilon\setminus\mathcal{K}')$, respectively. This implies that $\rm X(P)\setminus K'$ is connected by arcs, for $\rm\widetilde{R}^{-1}(\Upsilon\setminus\mathcal{K}')$ is connected by arcs. To prove our claim we suppose, without loss of generality, that $\rm d(\rho\circ\hat{\pi}(\eta_j),Ver(P))<\epsilon$, for $\rm j=1,2$. Let
\begin{equation}
	\label{moco}
	\rm
	\gamma_j:[0,L]\longrightarrow \rm\mathbf{S}^2(P)\setminus K_{\epsilon},\hspace{5mm}L\in\mathbf{N}, \hspace{1mm} j=1,2
\end{equation}
be a parameterization for a simple loop passing through $\rm\hat{\pi}(\eta_j)$. Up to a change of basis for $\rm \pi_1(\mathbf{S}^2(P),*)$, we can suppose that the image of each $\rm\gamma_j$ is in the homotopy class of  $\rm B_{i_j}^L$, for $\rm j=1,2$.  Let $\rm \widetilde{\gamma_j}$ be a lift of $\rm \gamma_j$ to $\rm X(P)\setminus K'$ such that $\rm \widetilde{\gamma_j}(0)=\eta_j$, $\rm j=1,2$. Given that the matrix $\rm D\circ hol(B_{i_j})$ is a rotation by an irrational multiple of $\pi$, the lift $\rm \widetilde{\gamma_j}$ is never closed and the parameter $\lambda'$ presents no resonances of the form $\rm ne_j$, for $\rm n\in\mathbf{N}$ and $\rm j=1,2$. Therefore, $\rm\widetilde{R}(\widetilde{\gamma_j})$ is a path on the graph $\Upsilon$ without cycles passing through at least $\rm L$ vertices. The compact set $\rm\mathcal{K}'\subset\Upsilon$ contains finitely many vertices. Then, for $\rm L$ large enough, the point $\rm\widetilde{R}\circ\widetilde{\gamma_j}(L)$ is contained in $\Upsilon\setminus\mathcal{K}'$, for both $\rm j=1,2$. For such $\rm L$, it is then sufficient to consider $\rm\eta_j'=\widetilde{\gamma_j}(L)$.\\
\\
\underline{Case A.2: Rank $\rm G(\lambda')=1$}. As in the preceding case, we suppose that $\rm d(\rho\circ\hat{\pi}(\eta_j),Ver(P))<\epsilon$, for $\rm j=1,2$. Let
\begin{equation}
	\label{toco}
	\rm
	\delta_j:\mathbf{R}\longrightarrow\mathbf{S}^2(P)\setminus K_{\epsilon}
\end{equation}
be a $\mathbf{Z}$-covering of a simple loop passing through $\rm\hat{\pi}(\eta_j)$ and such that $\rm\delta_j(n)=\hat{\pi}(\eta_j)$, for all $\rm n\in\mathbf{Z}$, $\rm j=1,2$. As in the preceding case, we can suppose that, for each $\rm j=1,2$ the image of $\rm [0,L]$ via $\rm\delta_j$ is in the homotopy class of $\rm B_{i_j}^L$ and that $\rm D\circ hol(B_{i_j})$ is a rotation by an angle $\rm\lambda_{i_j}$ that is not a rational multiple of $\pi$. Let $\rm\widetilde{\delta_j}$ denote the lift of $\rm\delta_j$ to $\rm X(P)\setminus K'$ satisfying $\rm\widetilde{\delta_j}(0)=\eta_j$. We claim that image of the composition $\rm\widetilde{R}\circ\widetilde{\delta_j}$ 'makes progress' in the Cayley graph $\Upsilon$. More precisely, that the image of this composition is a regular subgraph $\rm\Upsilon'_j$  of $\Upsilon$  isomorphic to the Cayley graph of $\mathbf{Z}$. Indeed, assume $\rm\widetilde{R}(\eta_j)\in\widetilde{R}(\hat{\pi}^{-1}(*))$. Then, the vertices of $\rm\Upsilon'_j$ are $\rm \widetilde{R}\circ\widetilde{\delta_j}(n)$, for $\rm j=1,2$ and $\rm n\in\mathbf{Z}$. Given that $\rm \lambda_{i_1}$ and  $\rm \lambda_{i_2}$ are not rational multiples of $\pi$,  the parameter $\lambda'$ presents no resonances of the form $\rm ne_j$, for $\rm n\in\mathbf{N}$ and $\rm j=1,2$. This implies that all vertices in $\rm \{\rm\widetilde{R}\circ\widetilde{\delta_j}(n)\}_{n\in\mathbf{Z}}$ are different, for each $\rm j=1,2$ and the claim follows. In the next figure, we illustrate with normal and dotted arrows the 'progress' of the graphs $\rm\Upsilon'_j$  in $\Upsilon$ in the case of an isoceles triangle. Here $\rm Res(\lambda')$ is generated by $(1,-1)$ and integers in the figure show the identifications on the vertices required to obtain $\Upsilon$.
\begin{center}
\includegraphics[scale=0.3]{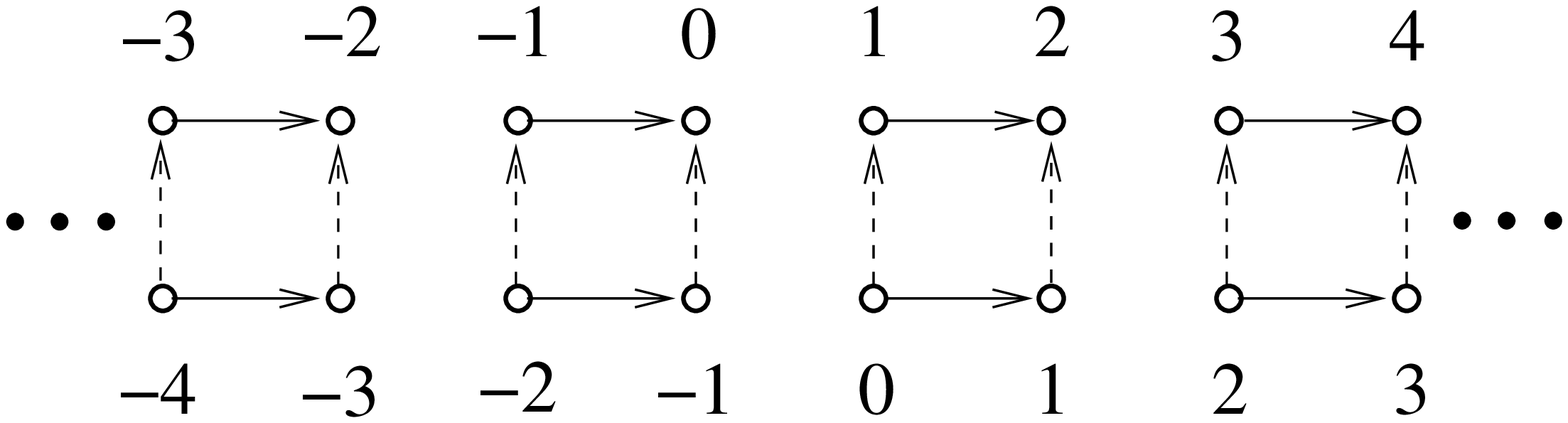}\\
Figure 4: graphs $\rm\Upsilon_1'$ and $\rm \Upsilon_2'$ (dotted arrows) in $\Upsilon$.
\end{center}
\indent When the rank of $\rm G(\lambda')$ is equal to one, there is a finite subgraph $\mathcal{K}'\subset\Upsilon$ containing $\mathcal{K}$ such that $\rm\Upsilon\setminus\mathcal{K}'$ has two connected components $\Xi_1$ and $\Xi_2$. The graph $\mathcal{K}'$ can be chosen so that, for each $\rm j=1,2$, $\rm\Upsilon_j\setminus \mathcal{K}' $ has two connected components as well, namely $\rm\Upsilon_j\cap \Xi_1$ and $\rm\Upsilon_j\cap \Xi_2$. Proceeding with the covering $\rm\delta_{j}$ as we did with the loop $\rm\gamma_j$, $\rm j=1,2$ in the preceding case, we can connect the points $\eta_1$ and $\eta_2$, respectively, through an arc with points $\eta_1'$ and $\eta_2'$ in $\rm X(P)\setminus K'$ such that $\rm\widetilde{R}(\eta_1')$ and $\rm\widetilde{R}(\eta_2')$ are contained in the same connected component of $\rm\Upsilon\setminus\mathcal{K}'$, say $\rm\Xi_1 $. The points $\eta_1'$ and $\eta_2'$ can then be connected through an arc within $\rm\widetilde{R}^{-1}(\Xi_1)$,  and we obtain this way an arc connecting $\eta_1$ to $\eta_2$.\\
\end{proof}
 \textbf{The surface $\rm X(P)$ has only one end. Case (B)}. We now turn to the case where
not all angles of $\rm P$ are irrational multiples of $\pi$.
Let $\rm Ver_{\mathbf{Q}}(P)$ be the set of vertices of $\rm P$ at which the interior angle is a rational multiple of $\pi$ and $\rm Ver_I(P):= Ver(P)\setminus Ver_{\mathbf{Q}}(P)$. We denote by $\rm\overline{\mathbf{S}^2(P)}$ the punctured sphere obtained by adding all rational vertices of $\rm P$ to $\rm\mathbf{S}^2(P)$. In the following diagram 
\begin{equation}
	\label{diagr}
	\rm
\xymatrix{\rm
X(P_0) \ar[d]^{\hat{\pi}} \ar@{^{(}->}[r]^{\chi_2} &\rm X(P)\ar[d]^{\tilde{\pi}}\\
\rm \mathbf{S}^2(P) \ar@{^{(}->}[r]^{\chi_1}  &\rm\overline{\mathbf{S}^2(P)}}
\end{equation}
 $\rm\chi_1$ and $\rm \chi_2$ are the natural embeddings defined when "adding rational vertices" and $\tilde{\pi}$ is the corresponding branched covering of  $\rm X(P)$ over the punctured sphere $\rm\overline{\mathbf{S}^2(P)}$. Let $\rm K\subset X(P)$ be a compact set. Our strategy will be the same, that is, to show that there is exist a compact set $\rm K\subset K'(K)\subset X(P)$ such that $\rm X(P)\setminus K'$ is connected. \\
 \indent For $0<\epsilon$ sufficiently small, the set  $\rm\hat{\pi}\circ\chi_2^{-1}(K)$ is contained in the complement of an $\epsilon$-neighborhood of the set of irrational vertices in $\rm\mathbf{S}^2(P)$. Abusing notation, we write
  \begin{equation}
	\label{kepsilon1}
	\rm
	K_{\epsilon}=\{t\in\mathbf{S}^2(P)\hspace{1mm}\mid\hspace{1mm} d(\rho(t),Ver_I(P))\geq\epsilon\},
\end{equation}
 and define $\rm U_{\epsilon}$ as in (\ref{uepsilon}). For every $\rm\xi\in\hat{\pi}^{-1}(U_{\epsilon})$ we denote by $\rm\widetilde{ U_{\epsilon}(\xi)}$ a connected lift of $\rm U_{\epsilon}$ to $\rm X(P_0)$. For every choice $\rm\Omega=\{\xi_1,\ldots,\xi_s\}\subset\hat{\pi}^{-1}(U_{\epsilon}) $ define $\rm K'(\Omega)$ to be the closure in $\rm\rm X(P)$ of $\rm \chi_2(\cup_{j=1}^s\widetilde{ U_{\epsilon}(\xi_j)})$. Choose $\Omega$ such that
 \begin{description}
\item[(i) ] The compact set $\rm K$ is contained in $\rm K'(\Omega)$.\\ 
\item[(ii)] For every simple loop $\rm\gamma\subset\tilde{\pi}(K'(\Omega))$ around a rational vertex $\rm v$ and every connected lift $\rm \widetilde{\gamma}$ of $\rm \gamma$ to $\rm X(P)$ satisfying $\rm \widetilde{\gamma}\subset K'(\Omega)$, one has that  $\rm \widetilde{\gamma}$ is contractible within $\rm K'(\Omega)$ to a point over $\rm v$. 
\end{description}
\indent For such a choice of $\Omega$, we affirm that $\rm X(P)\setminus K'(\Omega)$ is connected by arcs. Let $\rm q_1,q_2\in X(P)\setminus K'(\Omega)$. Property (ii) implies that $\rm q_1$ and $\rm q_2$ can be connected within $ \rm X(P)\setminus K'(\Omega)$ respectively through an arc with points $\rm q_1'$ and $\rm q_2'$ such that each projection $\rm \tilde{\pi}(q_j')$, $\rm j=1,2$ lies in the interior of an $\epsilon$-neighborhood in $\rm\overline{\mathbf{S}^2(P)}$ of the set of irrational vertices $\rm Ver_I(P)$.\\
\indent Define $\rm\alpha_j:=\chi_2^{-1}(q_j')$, $\rm j=1,2$. We claim that these points can be connected through an arc $\gamma$ within $\rm X(P_0)\setminus\chi_2^{-1}(K'(\Omega))$. Then $\rm\chi_2(\gamma)$ connects $\rm q_1'$ with $\rm q_2'$, and this completes the proof of Case (B).\\
\indent To prove our claim, chose a basis $\rm\Gamma:=\{B_{i_j}\}_{j=1}^{N-1}$ for the fundamental group of $\rm\mathbf{S}^2(P)$ such that, for each $\rm j=1,2$ there exist a loop $\rm B_{i_j}$ in $\Gamma$ whose index with respect to $\rm \hat{\pi}(\alpha_j)$ is equal to 1. Given that  $\rm \tilde{\pi}(q_j')$ lies in the interior of an $\epsilon$-neighborhood of $\rm Ver_I(P)$, each matrix $\rm D\circ hol(B_{i_j})$ represents a rotation by an angle  which is not a rational multiple of $\pi$.  Let $\lambda'$ be the set of angle parameters defined by our choice of basis $\Gamma$. As in the preceding situation, we consider two cases: (B.1) Rank $\rm G(\lambda')\geq 2$ and  (B.2) Rank $\rm G(\lambda')=1$. \\ 
\indent  Set $\rm\mathcal{K}=\widetilde{R}(\chi_2^{-1}(K'(\Omega)))$. For case (B.1) the graph $\Upsilon$ has only one end.  Given that the angles of the rotations $\rm D\circ hol(B_{i_j})$ are not rational multiples of $\rm \pi$ we can proceed as in case (A.1), Lemma 2. In other words, find a simple loop $\rm\gamma_j$ passing through $\rm\hat{\pi}(\alpha_j)$ for which no connected lift $\rm\widetilde{\gamma_j}$ to $\rm X(P_0)$ is closed. Then, 'making progress' in the Cayley graph $\Upsilon$, reach a points $\alpha'_1$ and $\alpha'_2$ in $\rm\widetilde{R}^{-1}(\Upsilon\setminus\mathcal{K}')$, that can be connected through an arc. 
Here, $\mathcal{K}\subset\mathcal{K}'\subset\Upsilon$ denotes a finite graph such that $\Upsilon\setminus \mathcal{K}'$ is connected. For (B.2)  set $\rm\mathcal{K}=\widetilde{R}(\chi_2^{-1}(K'(\Omega)))$ as well and proceed as in case (A.2), Lemma 2.\\

\textbf{The surface $\rm X(P)$ has infinite genus}. Here our strategy is to find a sequence of domains $\rm \{D_j\}_{j\in\mathbf{Z}}$ in $\rm X(P)$,  each homeomorphic to a torus to which we have removed a finite number of small discs and such that $\rm D_j\cap D_i=\emptyset$ whenever $\rm i\neq j$. Therefore, there exists no finite set $\rm \mathcal{C}\subset X$ of mutually non-intersecting simple closed curves with the property that $\rm X\setminus \mathcal{C}$ is a connected and planar surface.\\
\\
\emph{Remark}. When proving that $\rm X(P)$ has infinite genus, we \emph{do not} suppose that all interior angles of $\rm P$ are irrational multiples of $\pi$.\\
\\
\indent Let $\rm v$ and $\rm w$ denote vertices of $\rm P$ whose interior angles are not rational multiples of $\pi$. Denote these angles by $\rm\lambda_{i_1}\pi$ and $\rm\lambda_{i_2}\pi$ respectively. The boundary $\rm\partial P\setminus\{v,w\}$ has two connected components. Let $\rm S$ denote a connected component containing at least two sides of $\rm P$. Define
\begin{equation}
	\label{U}
	\rm 
	U_{\epsilon}:=\{ t\in \mathbf{S}^2(P)\hspace{1mm}\mid\hspace{1mm} d(\rho(t),Ver(P))\geq \epsilon\},
	\end{equation}
and let $\rm U:=U_{\epsilon}\setminus\rho^{-1}(S)$. If $\epsilon>0$ is sufficiently small, $\rm U$ can be identified, topologically, with the following figure (for the sake of simplicity, we present an image corresponding to a pentagon $\rm P$):
\begin{center}
\includegraphics[scale=0.2]{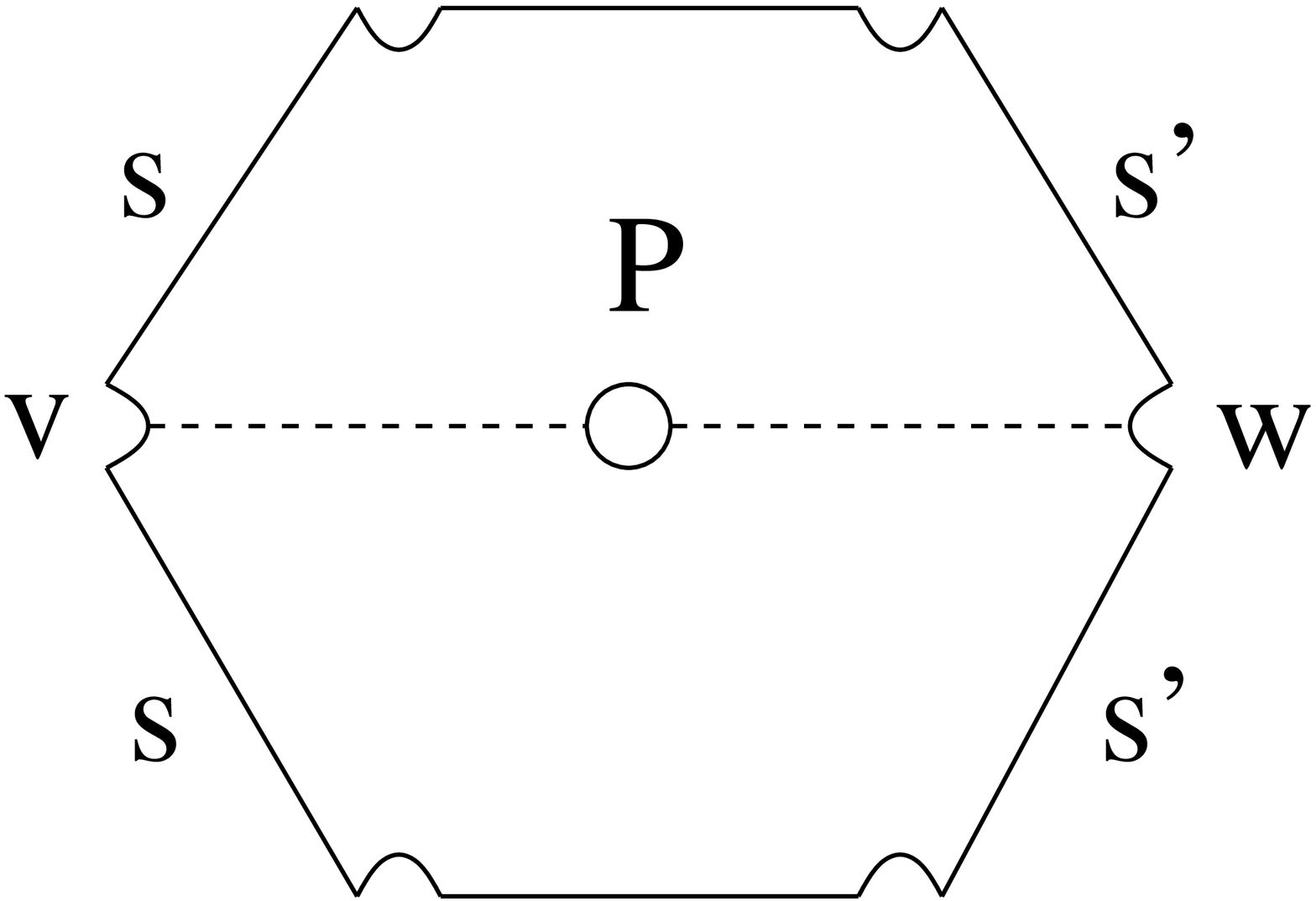}\\
Figure 5: The domain $\rm U\subset\mathbf{S}^2(P)$.
\end{center}
\indent In this figure $\rm s$ and $\rm s'$ denote the sides of $\rm P$ in $\rm S$ adjacent to the vertices $\rm v$ and $\rm w$ respectively. Let $\Gamma$ be a basis for the fundamental group of $\mathbf{S}^2(\rm P)$ containing two generators $\rm B_{i_1}$ and $\rm B_{i_2}$ in $\rm U$ whose holonomy is given by:
 \begin{equation}
  \label{rots1}
  \rm
  D\circ hol(B_{i_j})=
\begin{pmatrix}
\rm
  \cos(2\lambda_{i_j}\pi) &\rm -\sin(2\lambda_{i_j}\pi) \\
  \rm\sin(2\lambda_{i_j}\pi) & \rm\cos(2\lambda_{i_j}\pi)
\end{pmatrix}
\rm
\hspace{1cm} j=1,2.
 \end{equation}
 \indent Denote by $\rm U(m)$ a connected lift of $\rm U$ along a loop described by a word on the letters forming $\rm m\in\Gamma$. We claim that, despite the possible resonances, there is always a word $\rm m$ such that $\rm U(m)$ is homeomorphic to a torus from which we have removed a  finite number of discs. Since the angle $\rm\lambda_{i_1}\pi$ is not a rational multiple of $\pi$, for $\rm M\in\mathbf{N}$ sufficiently large, each domain of the form $\rm U(mB_{i_1}^{jM})$, $\rm j\in\mathbf{Z}\setminus 0$, contains a domain $\rm D_j$ homeomorphic to a torus and such that $\rm D_j\cap D_i=\emptyset$, whenever $\rm i\neq j$. \\
 \indent First, identify $\rm U$ and the loops $\rm B_{i_1}$ and $\rm B_{i_2}$ with the following figure 
 \begin{center}
\includegraphics[scale=0.28]{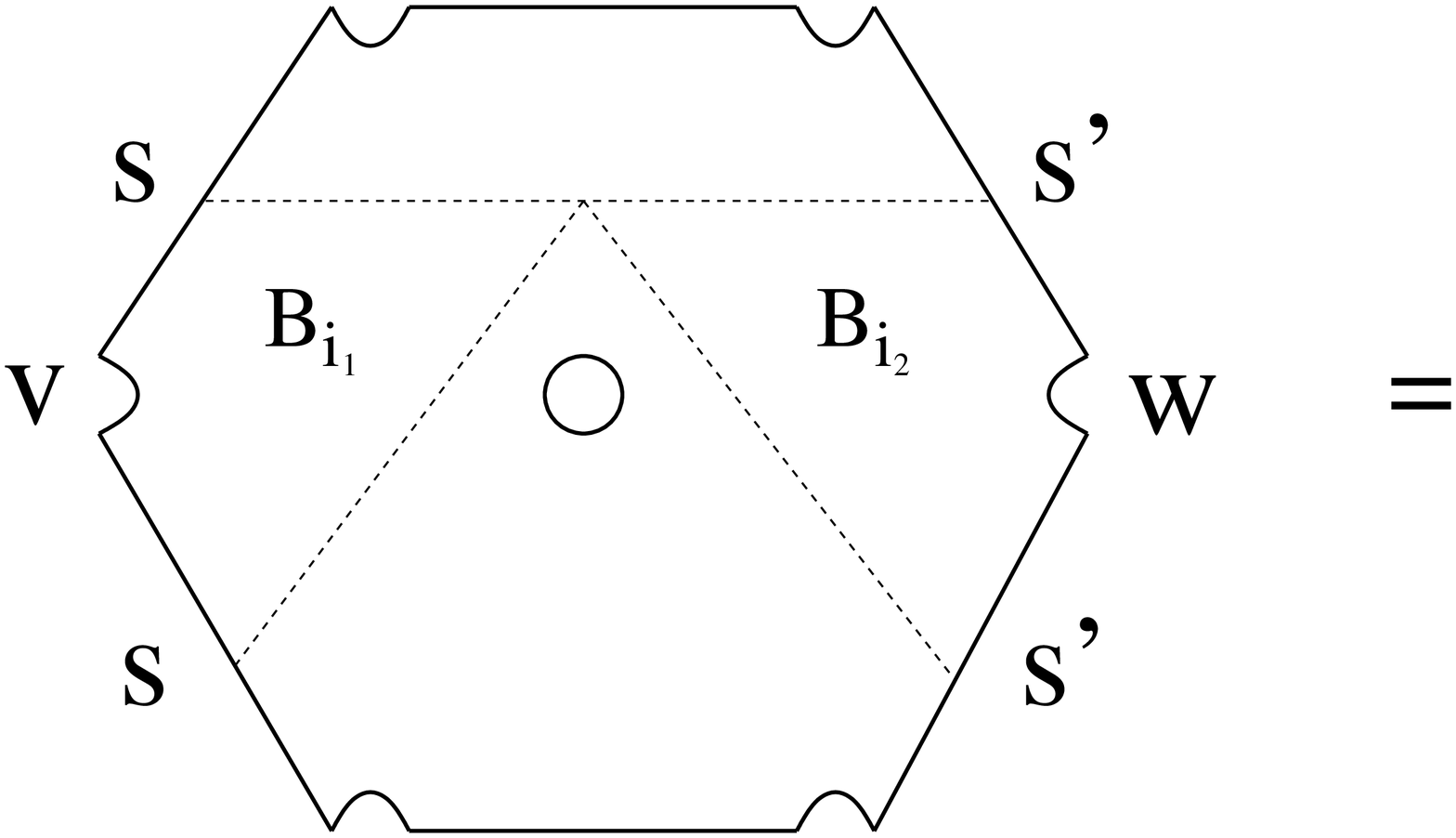}\hspace{4mm}
\includegraphics[scale=0.25]{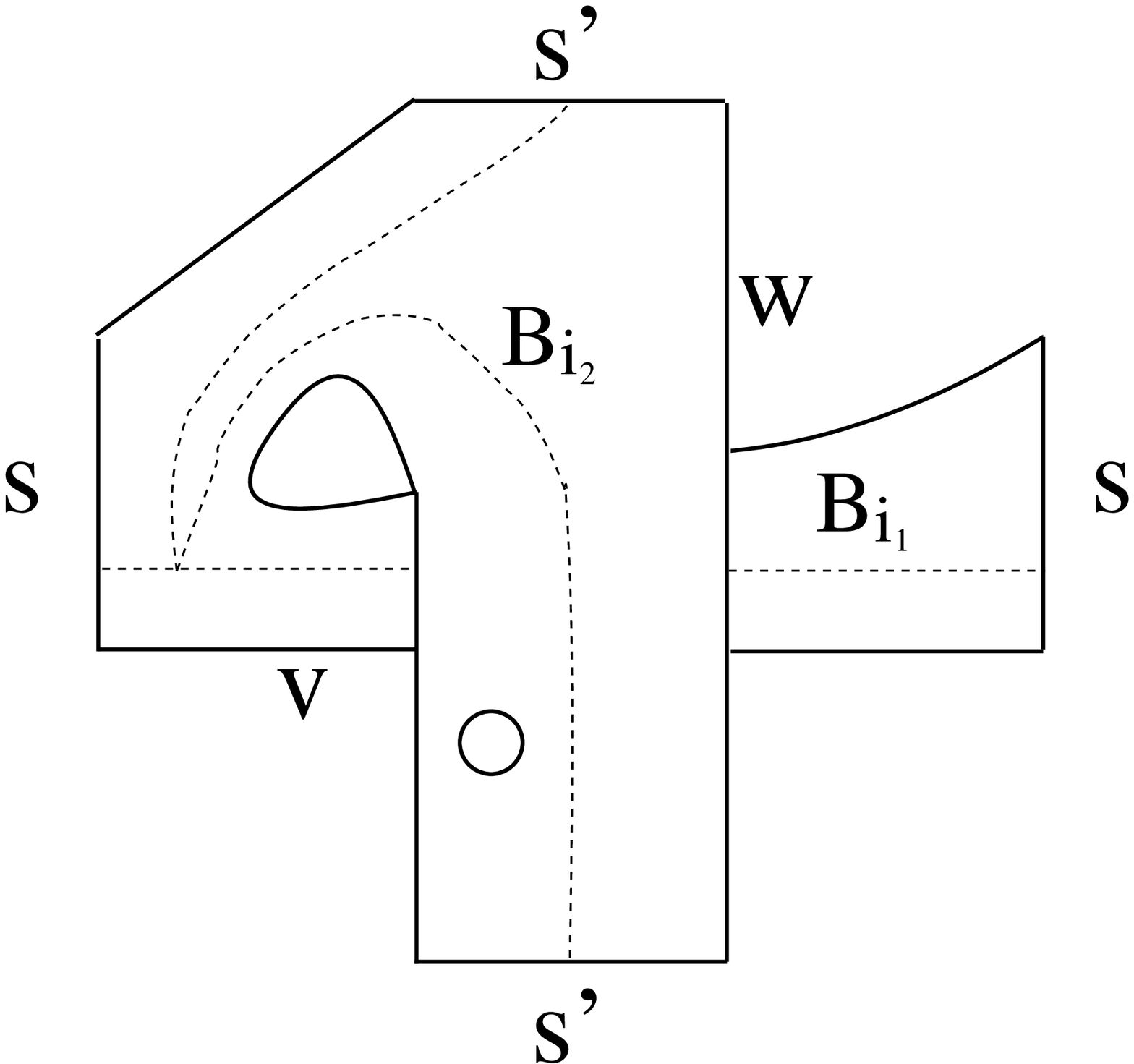}\\
Figure 6: The loops $\rm B_{i_1}$ and $\rm B_{i_2}$ in $\rm U$.
\end{center}
 \indent We denote by $\rm [m_1,m_2]=m_2^{-1}m_1^{-1}m_2m_1$ the commutator of two words in $\Gamma$. We proceed by cases, given by the resonances of the parameter $\lambda$.
 \begin{case}
 	\label{c1}
 We suppose that $\lambda$ is totally irrational. Then \linebreak $\rm U([B_{i_1}^{-1},B_{i_2}][B_{i_1},B_{i_2}])$ is homeomorphic to a torus from which we have removed a finite number of discs. We explain this step by step. A connected lift of a commutator $\rm [B_{i_1},B_{i_2}]$ to the covering $\rm X(P_0)$ via  $\hat{\pi}$ is a loop, for $\rm Aut(\hat{\pi})$ is abelian. Given that $\lambda$ has no non-trivial resonances, neither $\rm\lambda_{i_1}+\lambda_{i_2}$ nor $\rm \lambda_{i_1}-\lambda_{i_2}$ is an integer.  We deduce that the  lift $\rm U([B_{i_1},B_{i_2}])$ is of the form
 \begin{center}
 	\label{I}
\includegraphics[scale=0.2]{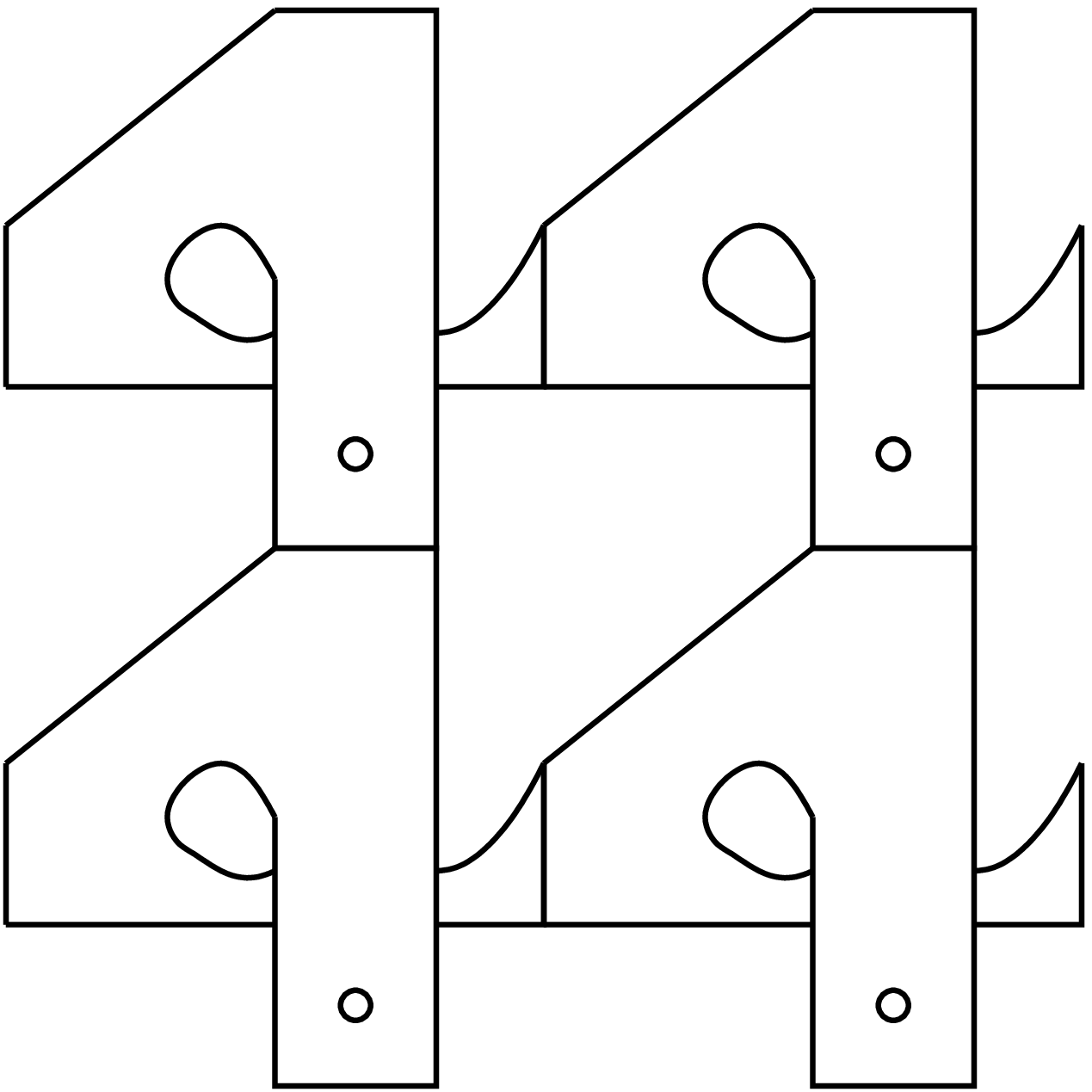}\hspace{15mm}\includegraphics[scale=0.3]{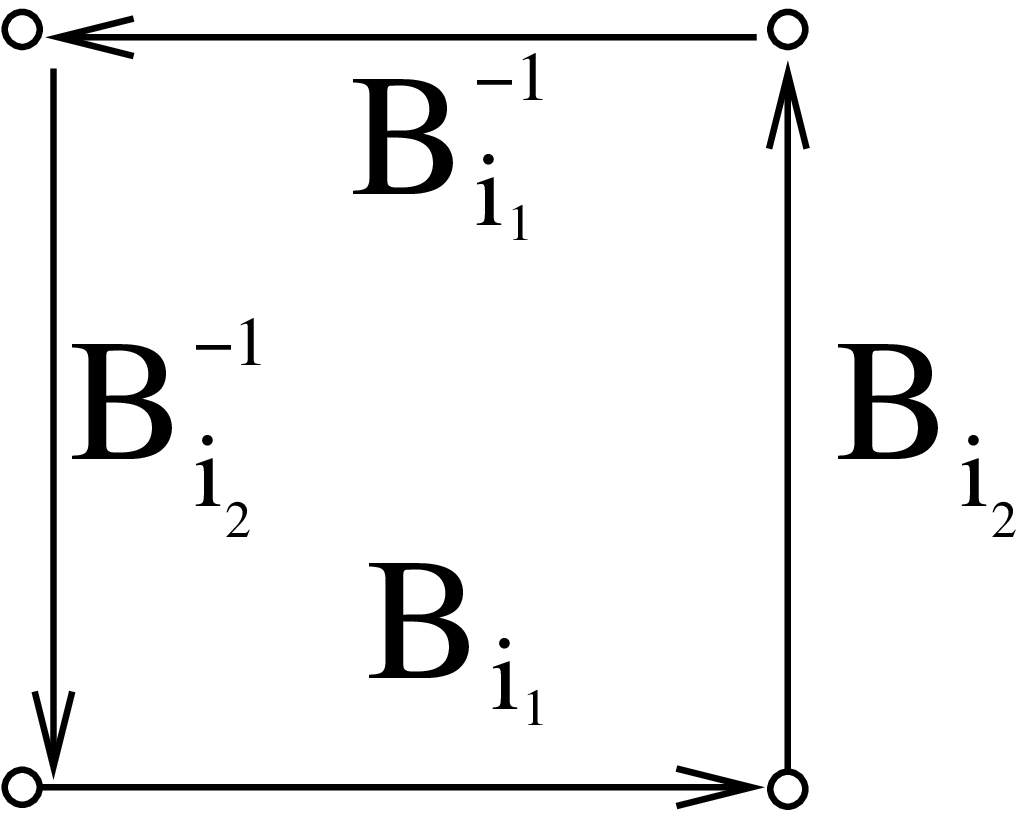}\\
Figure 7: The lift of the commutator $\rm[B_{i_1},B_{i_2}]$.

\end{center}
The commutativity of  $\rm Aut(\hat{\pi})$ and the lack of non-trivial resonances for $\lambda$, imply that the connected lift of the product\linebreak $\rm[B_{i_1}^{-1},B_{i_2}][B_{i_1},B_{i_2}]$ is "homeomorphic" to the capital letter $\rm B$. In terms of figures, $\rm U([B_{i_1}^{-1},B_{i_2}][B_{i_1},B_{i_2}])$ is given by
\begin{center}
\includegraphics[scale=0.2]{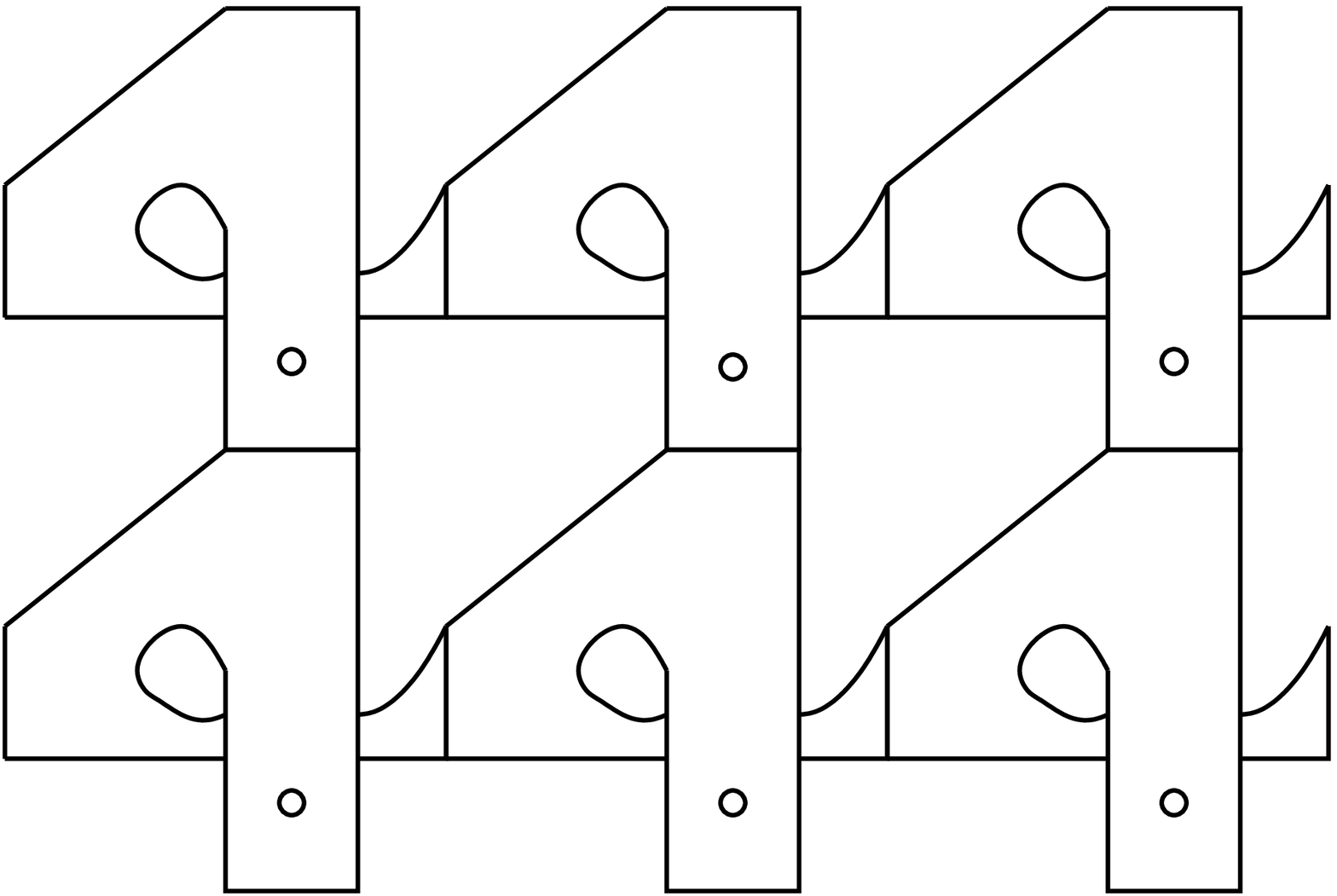}\hspace{15mm}\includegraphics[scale=0.3]{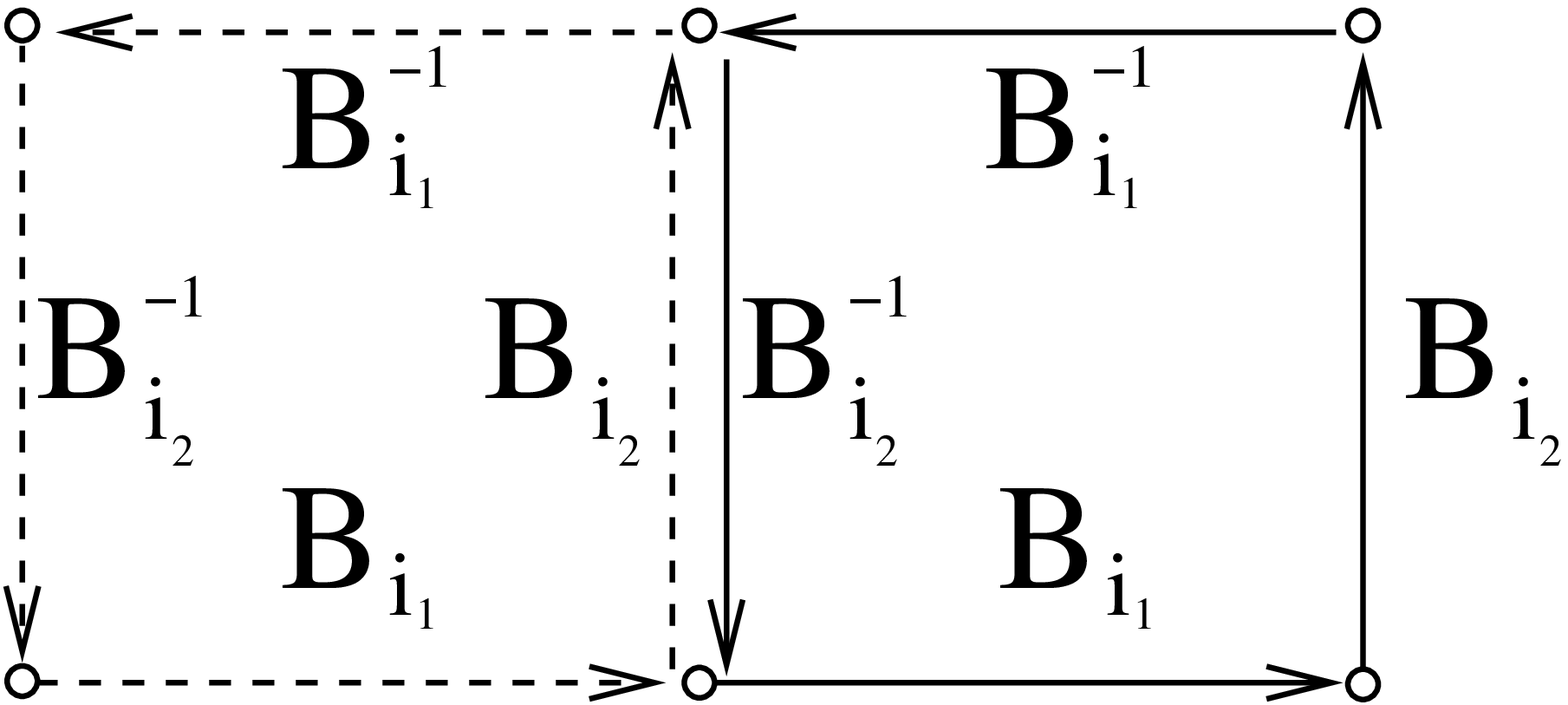}\\
Figure 8: The lift of the product of commutators $\rm[B_{i_1}^{-1},B_{i_2}][B_{i_1},B_{i_2}]$.

\end{center}
 Trivially, the domain $\rm U([B_{i_1}^{-1},B_{i_2}][B_{i_1},B_{i_2}])$ is obtained from the identification of two connected components A and A':
\begin{center}
\includegraphics[scale=0.25]{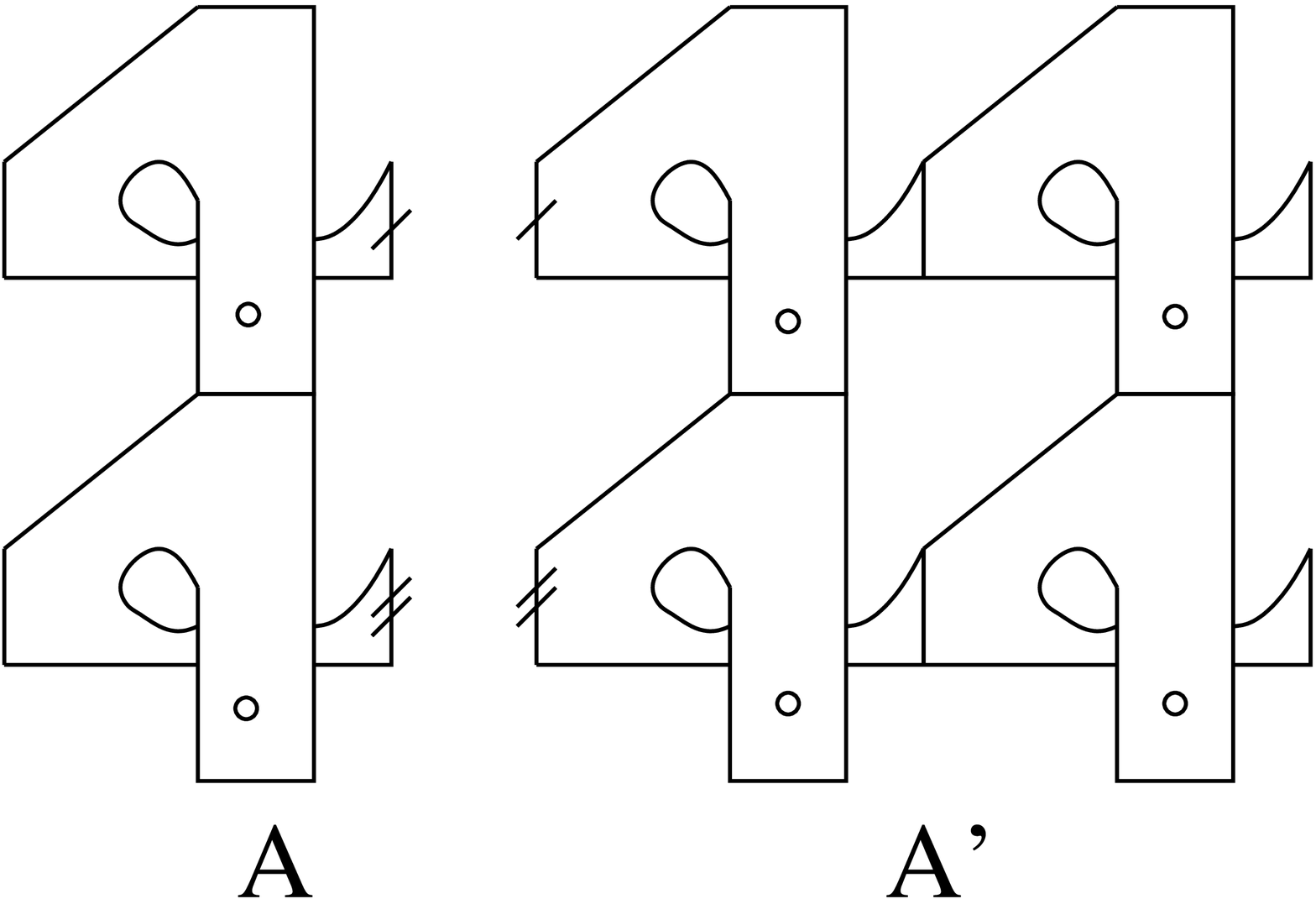}\\
Figure 9: The domain $\rm U([B_{i_1}^{-1},B_{i_2}][B_{i_1},B_{i_2}])$.
\end{center}
and this identification defines a torus to which one has removed a finite number of discs: 
 \begin{center}
\includegraphics[scale=0.25]{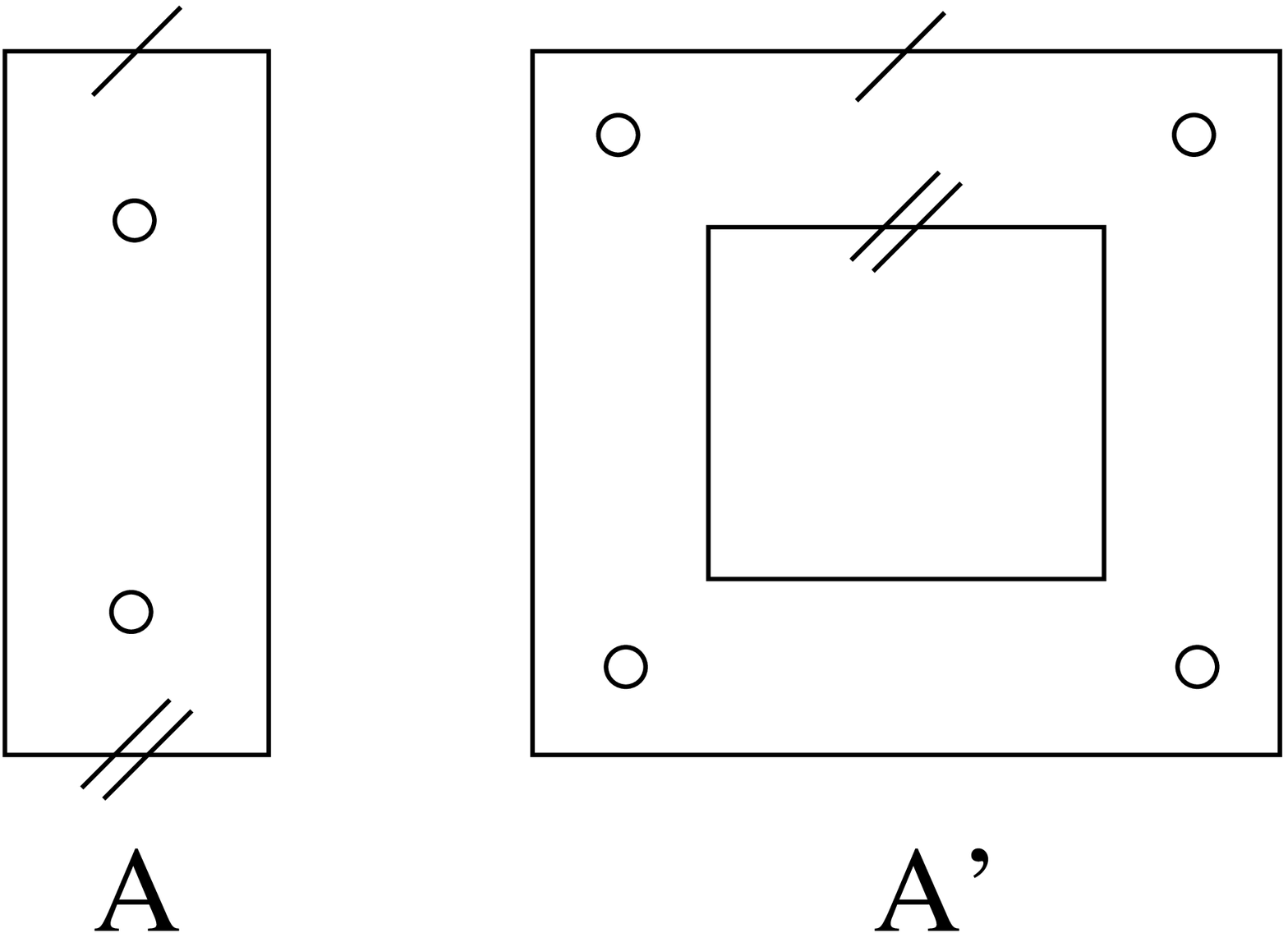}\\
Figure 10: The domain $\rm U([B_{i_1}^{-1},B_{i_2}][B_{i_1},B_{i_2}])$ is a "punctured" torus.
\end{center}
  \end{case}
 \begin{case}
In the preceding case no element of  $\rm R :=\{\rm\lambda_{i_1}+\lambda_{i_2}, \rm\lambda_{i_1}-\lambda_{i_2}, \rm2\lambda_{i_1}+\lambda_{i_2},\rm2\lambda_{i_1}-\lambda_{i_2}\}$ was an integer. This fact let us arrive to figure 8. Henceforth we deal with all possible subcases defined by non-empty subsets of  $\rm R(\mathbf{Z}):= R\cap \mathbf{Z}$. Define 
$$
(2.1)\hspace{1mm} \rm  R(\mathbf{Z})=\{2\lambda_{i_1}+\lambda_{i_2}\},\hspace{1mm}(2.2)\hspace{1mm} \rm  R(\mathbf{Z})=\{2\lambda_{i_1}-\lambda_{i_2}\},$$ $$\hspace{1mm}(2.3)\hspace{1mm}\rm  R(\mathbf{Z})=\{\lambda_{i_1}-\lambda_{i_2}\},\hspace{1mm}\text{and}\hspace{1mm}(2.4)\hspace{1mm}R(\mathbf{Z})=\{\lambda_{i_1}+\lambda_{i_2}\}
$$
\indent Remark that these are all possible non-empty subsets of $\rm  R\cap \mathbf{Z}$. Take as example, $\rm R(\mathbf{Z})=\{2\lambda_{i_1}+\lambda_{i_2},\lambda_{i_1}+\lambda_{i_2}\}$. This set of "resonances" will never occur, for the difference of its elements is equal to, say, $\rm\lambda_{i_1}$, which by hypotheses is not a rational number. All other possible cases not figuring in the list (2.1)-(2.4) are excluded by similar arguments. We now proceed explaining in detail how to obtain in subcase (2.1) a domain homeomorphic to a torus from which we have removed a finite number of discs.\\
\\
\textbf{Subcase (2.1)}. When neither $\rm\lambda_{i_1}+\lambda_{i_2}$, nor $\rm\lambda_{i_1}-\lambda_{i_2}$ are integers,  the  lift of the commutator $\rm U([B_{i_1},B_{i_2}])$ is given by figure 7. Given that $\rm 2\lambda_{i_1}+\lambda_{i_2}$ is an integer, a connected lift of the word $\rm B_{i_2}B_{i_1}^2$ to $\rm X(P_0)$ is a loop. We depict the domain $\rm U(B_{i_2}B_{i_1}^2[B_{i_1},B_{i_2}])$:
\begin{center}
\includegraphics[scale=0.2]{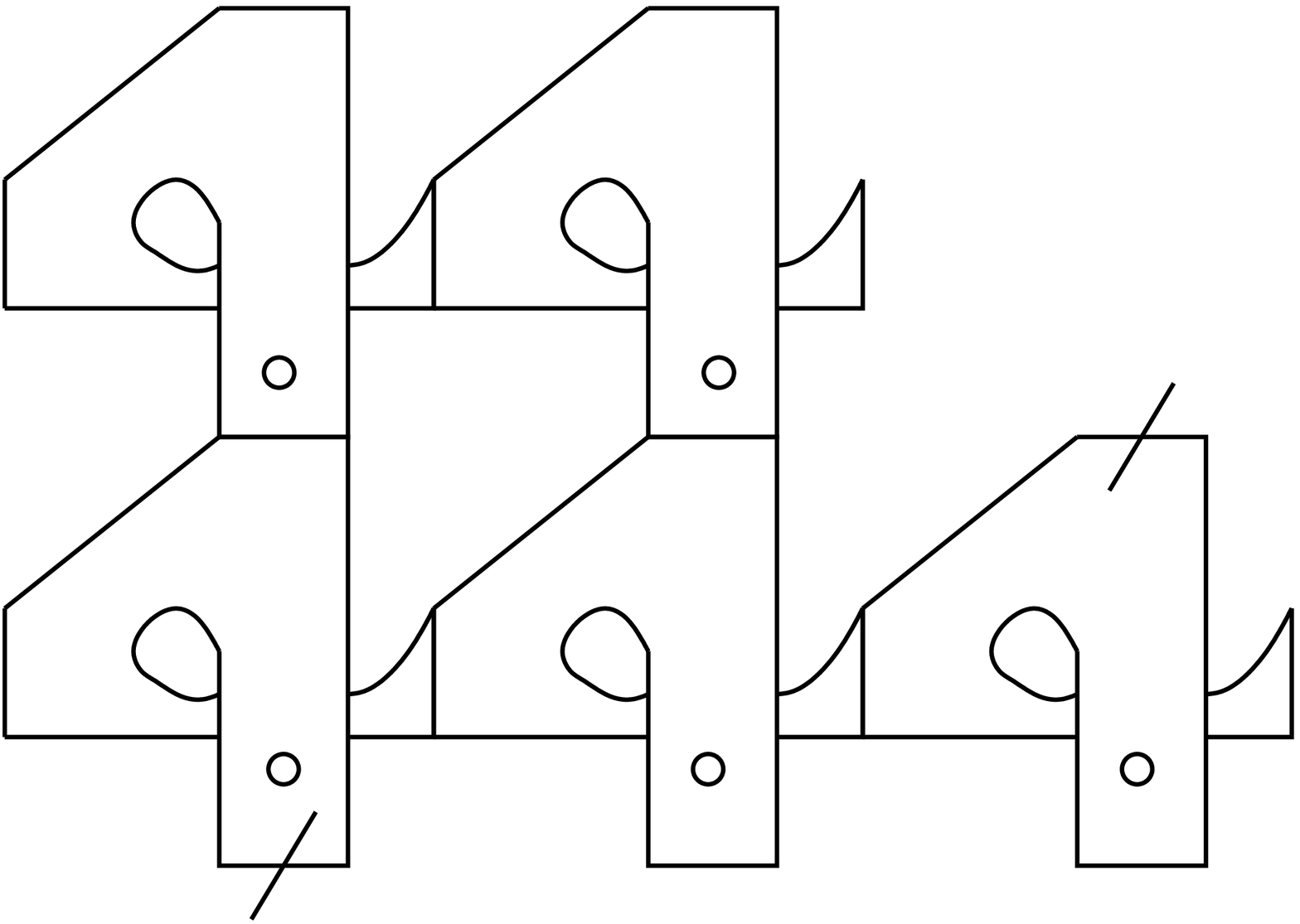}\hspace{15mm}\includegraphics[scale=0.3]{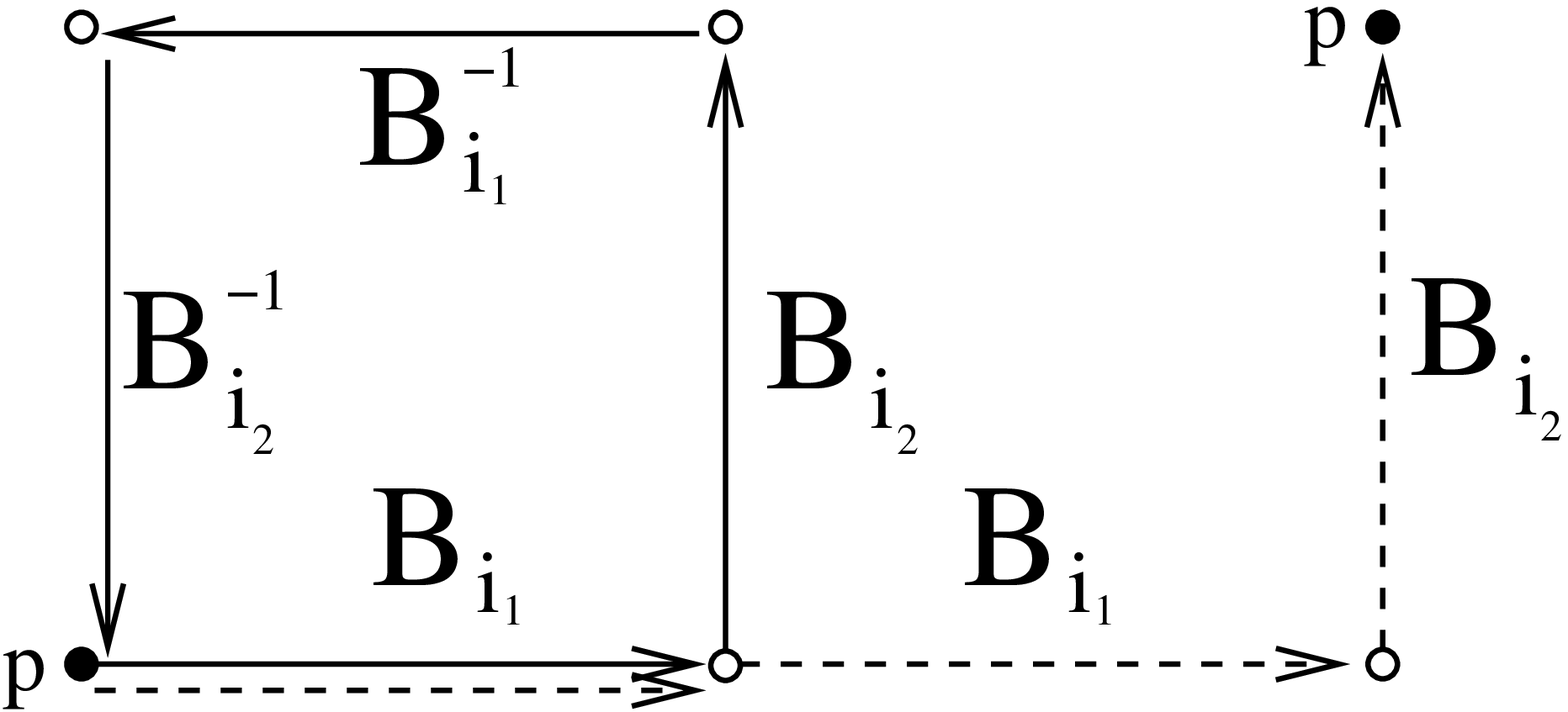}\\
Figure 11: The domain $\rm U(B_{i_2}B_{i_1}^2[B_{i_1},B_{i_2}])$. 
\end{center}
 \end{case}
In the preceding figure, points labeled with the letter $\rm p$ are identified. With a suitable homeomorphism, this figure can be simplified to: 
 \begin{center}
\includegraphics[scale=0.2]{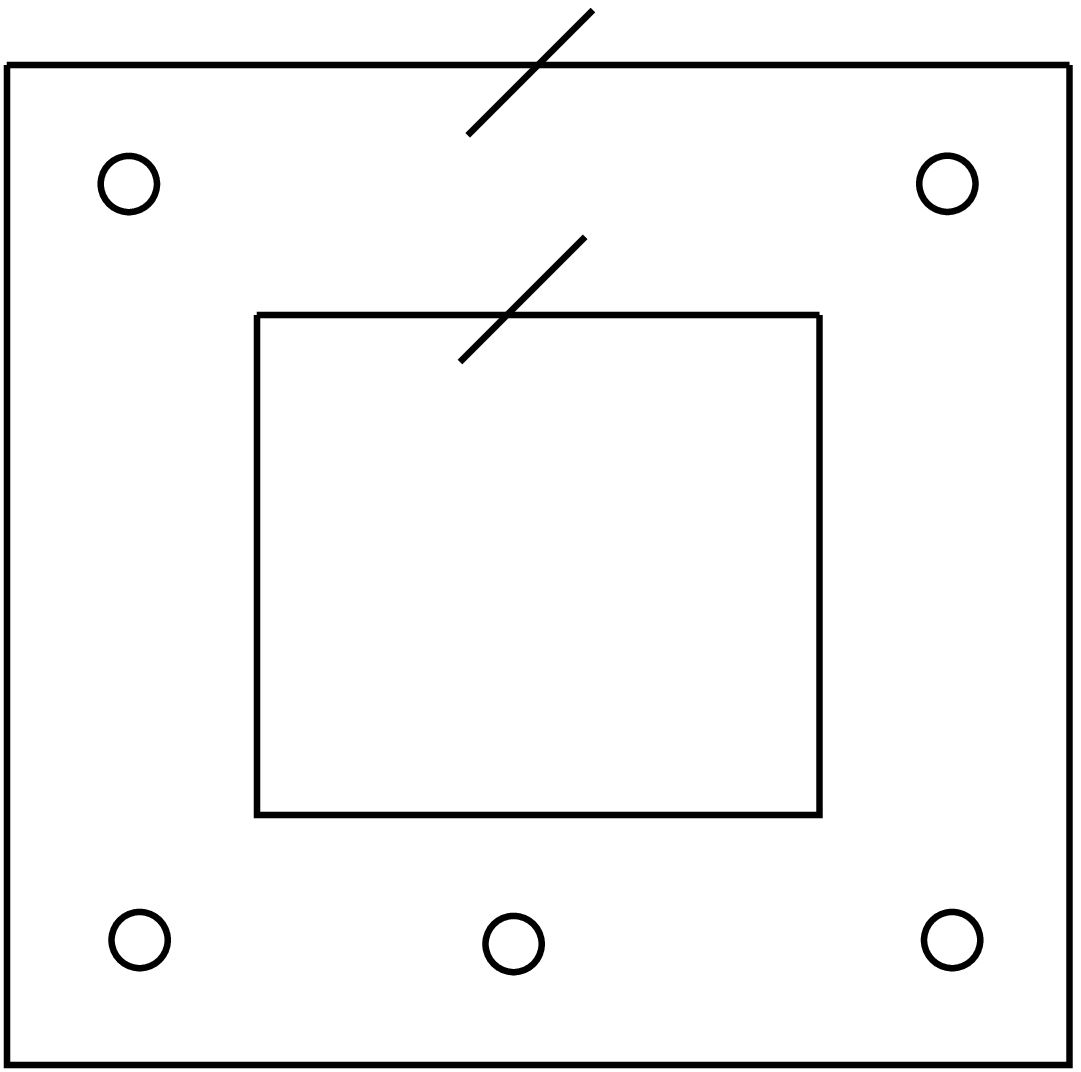}\\
Figure 12: Simplification of  $\rm U(B_{i_2}B_{i_1}^2[B_{i_1},B_{i_2}])$. 
\end{center}
which is clearly homeomorphic to a torus from which we have removed a finite number of discs. \\
\textbf{Subcase (2.2)}.  Proceeding as in the preceding subcase, we obtain that $\rm U(B_{i_1}^2B_{i_2}^{-1}[B_{i_1},B_{i_2}^{-1}])$ corresponds to the figure:
\begin{center}
\includegraphics[scale=0.2]{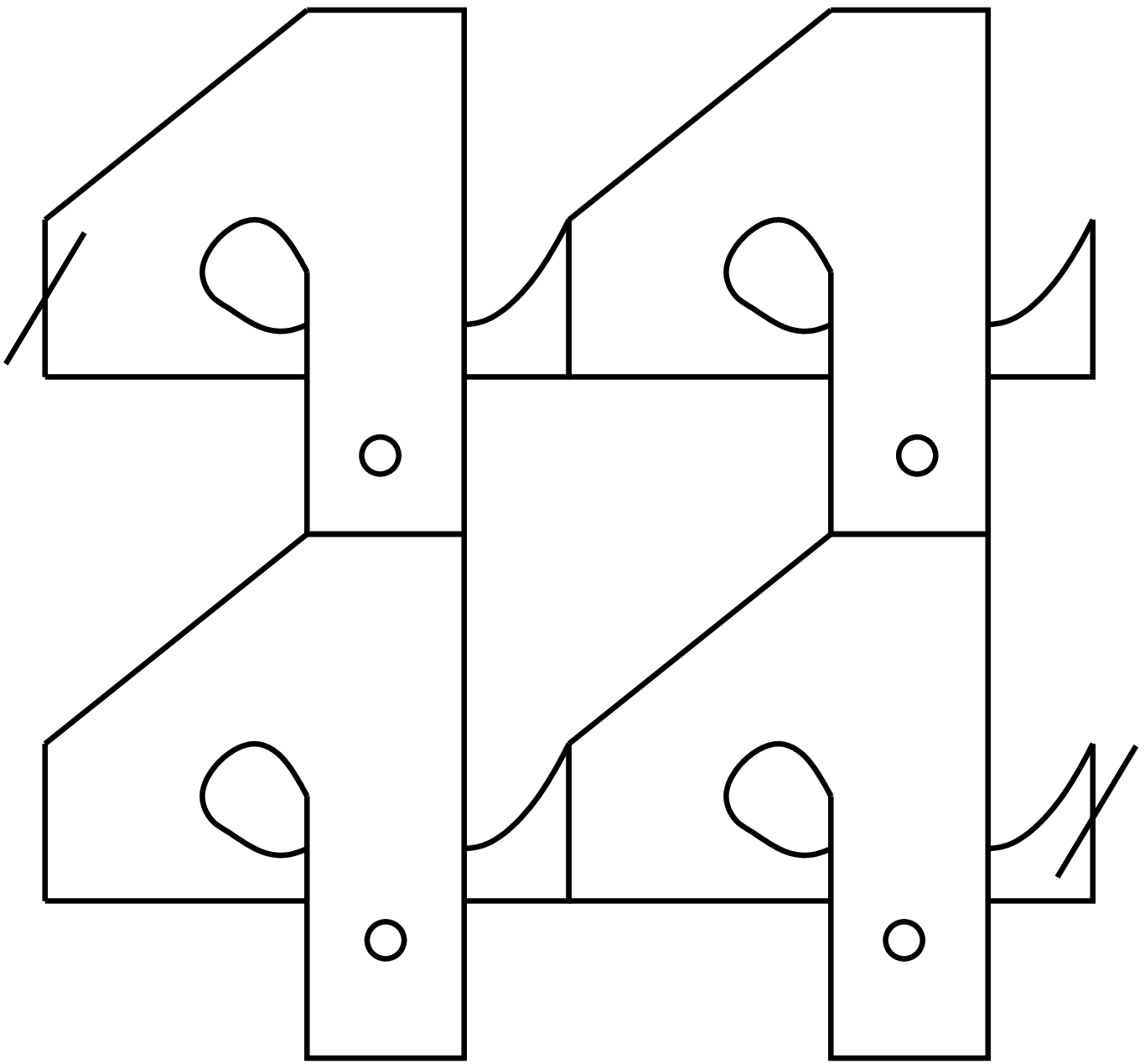}\hspace{15mm}\includegraphics[scale=0.3]{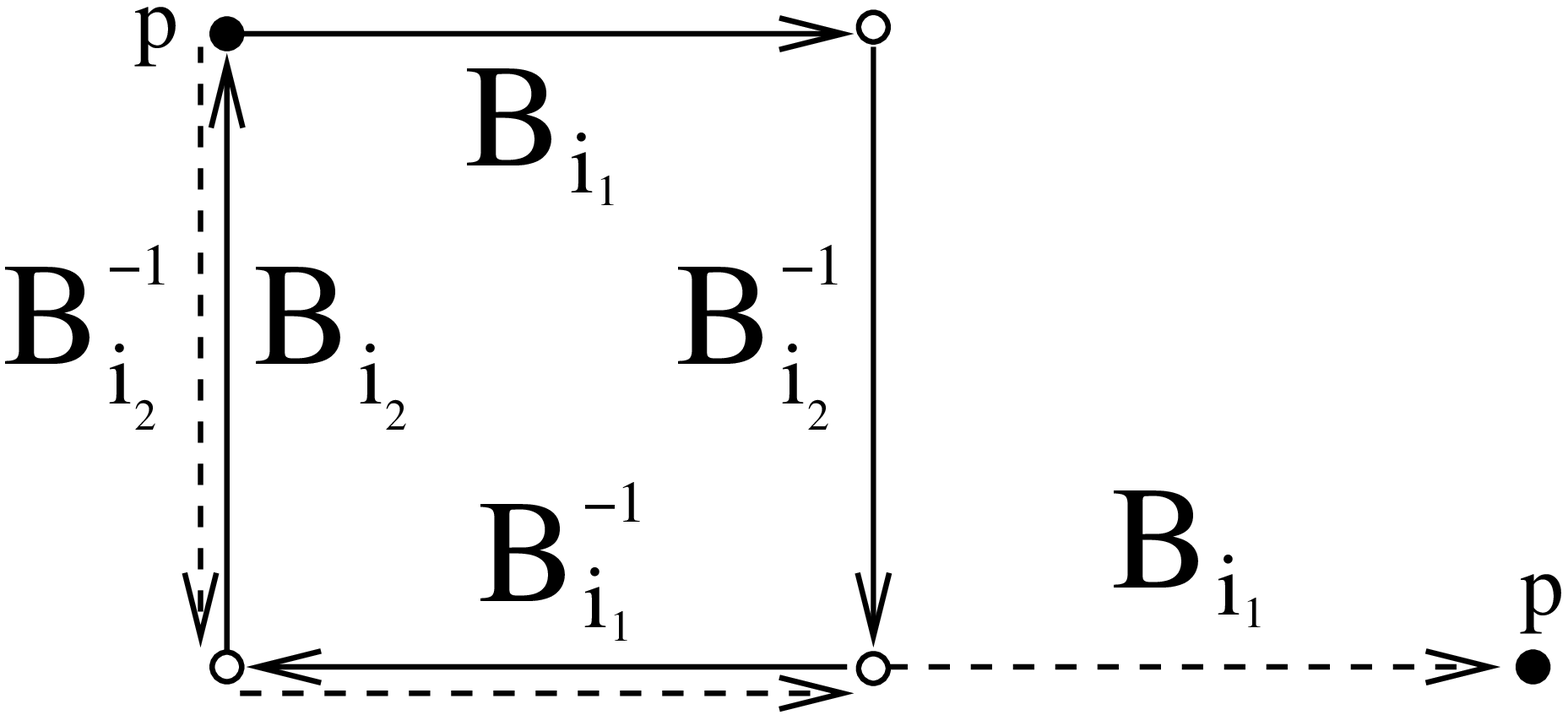}\\
Figure 13: The domain $\rm U(B_{i_1}^2B_{i_2}^{-1}[B_{i_1},B_{i_2}^{-1}])$. 
\end{center}
With a suitable homeomorphism, one can deform the preceding figure into a "punctured" torus:
 \begin{center}
\includegraphics[scale=0.2]{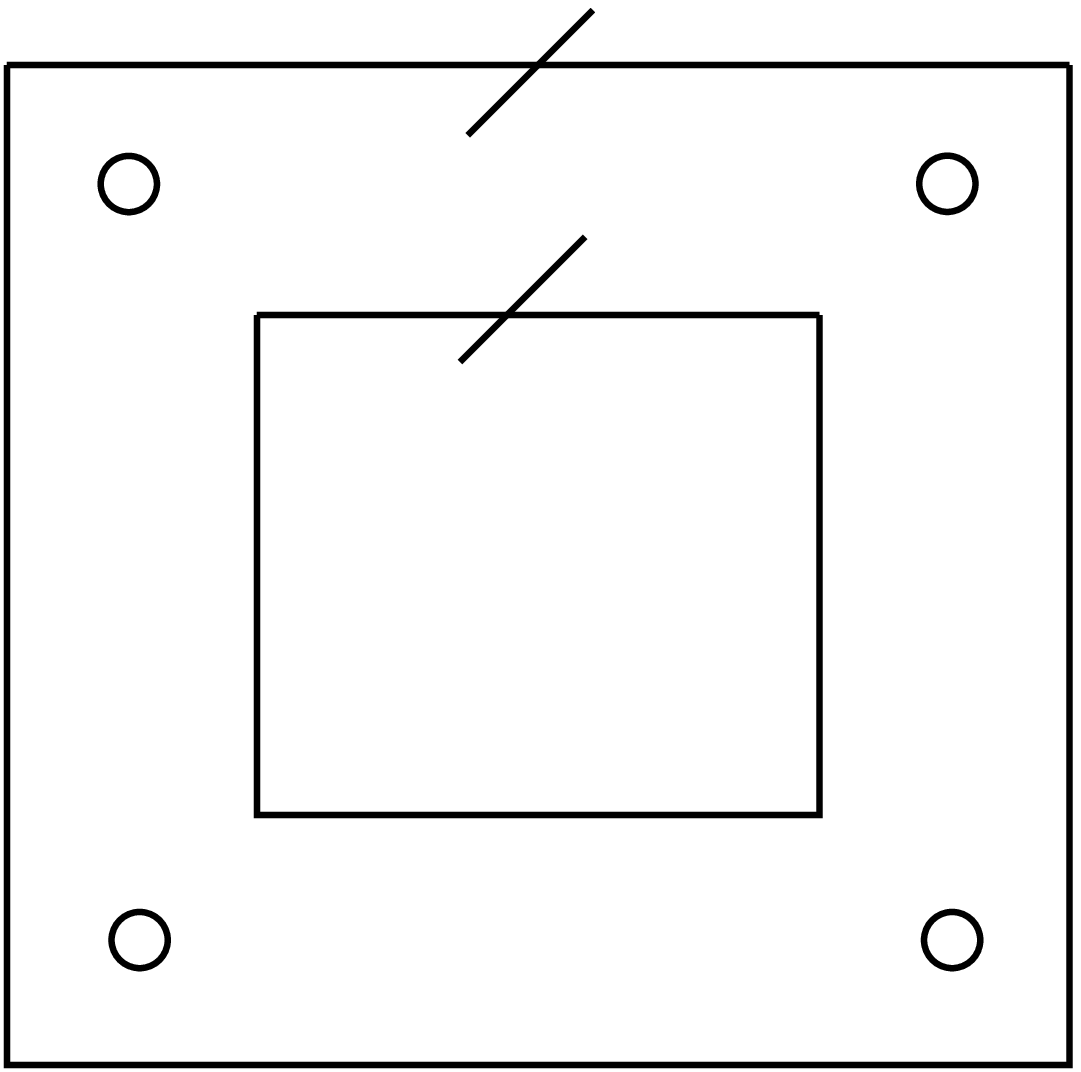}\\
Figure 14: Simplification of  $\rm U(B_{i_1}^2B_{i_2}^{-1}[B_{i_1},B_{i_2}^{-1}])$. 
\end{center}
\textbf{Subcase (2.3)}. In this situation the domain $\rm U(B_{i_2}^{-1}B_{i_1})$ is homeomorphic to a disc from which we have removed a finite number of discs:
\begin{center}
\includegraphics[scale=0.3]{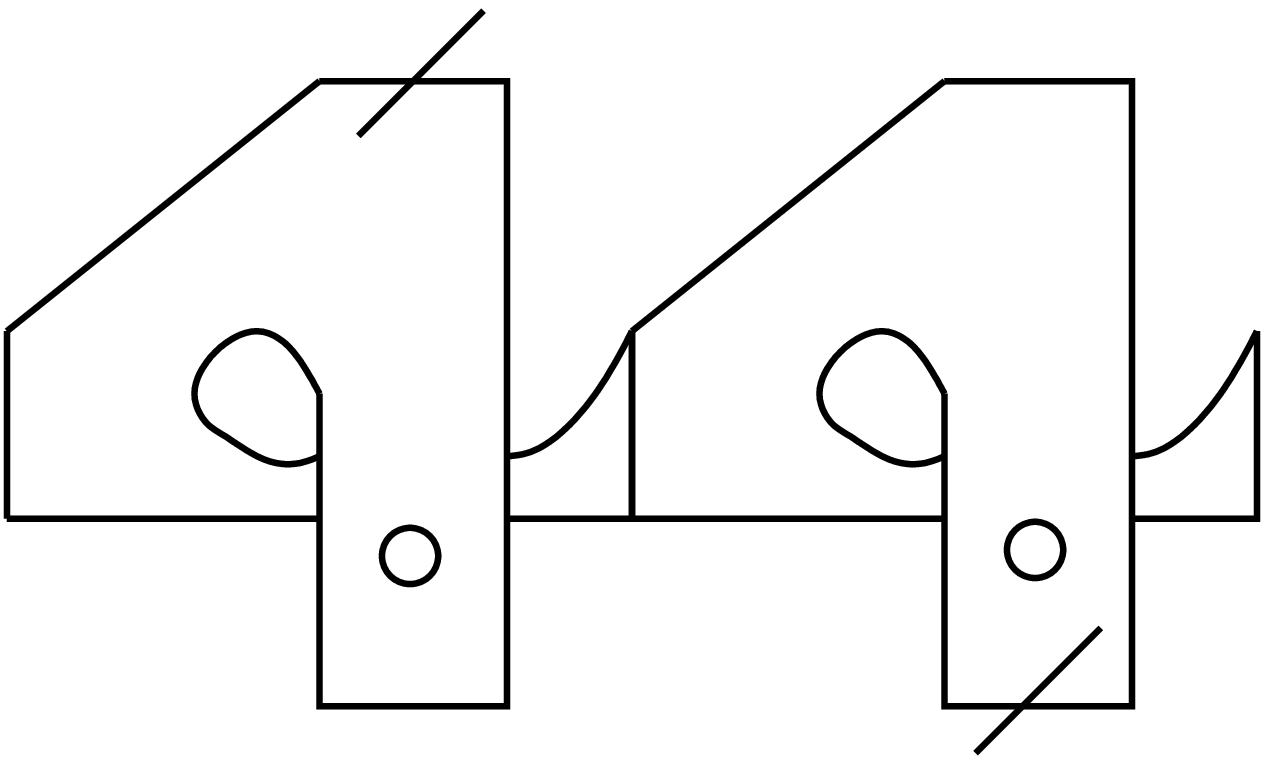}\hspace{15mm}\includegraphics[scale=0.3]{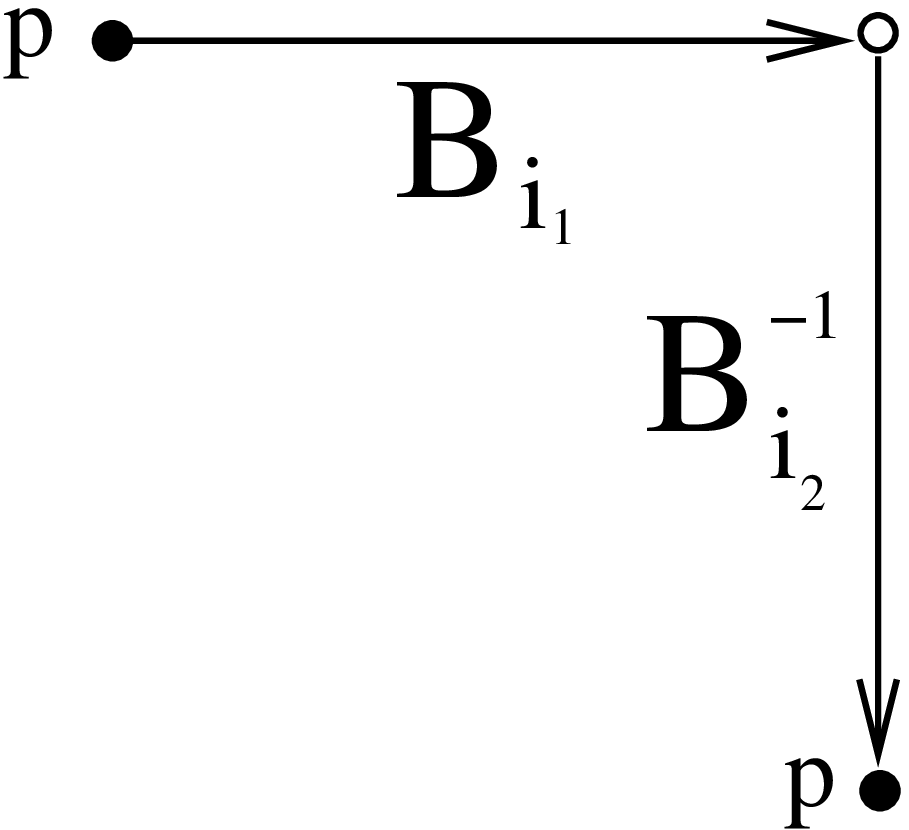}\\
Figure 15: The domain  $\rm U(B_{i_2}^{-1}B_{i_1})$.
\end{center}
\indent Then, $\rm U(B_{i_1}^{-1}B_{i_2}B_{i_1}B_{i_2}^{-1}B_{i_1})$ corresponds to the figure
\begin{center}
\hspace{1mm}\includegraphics[scale=0.3]{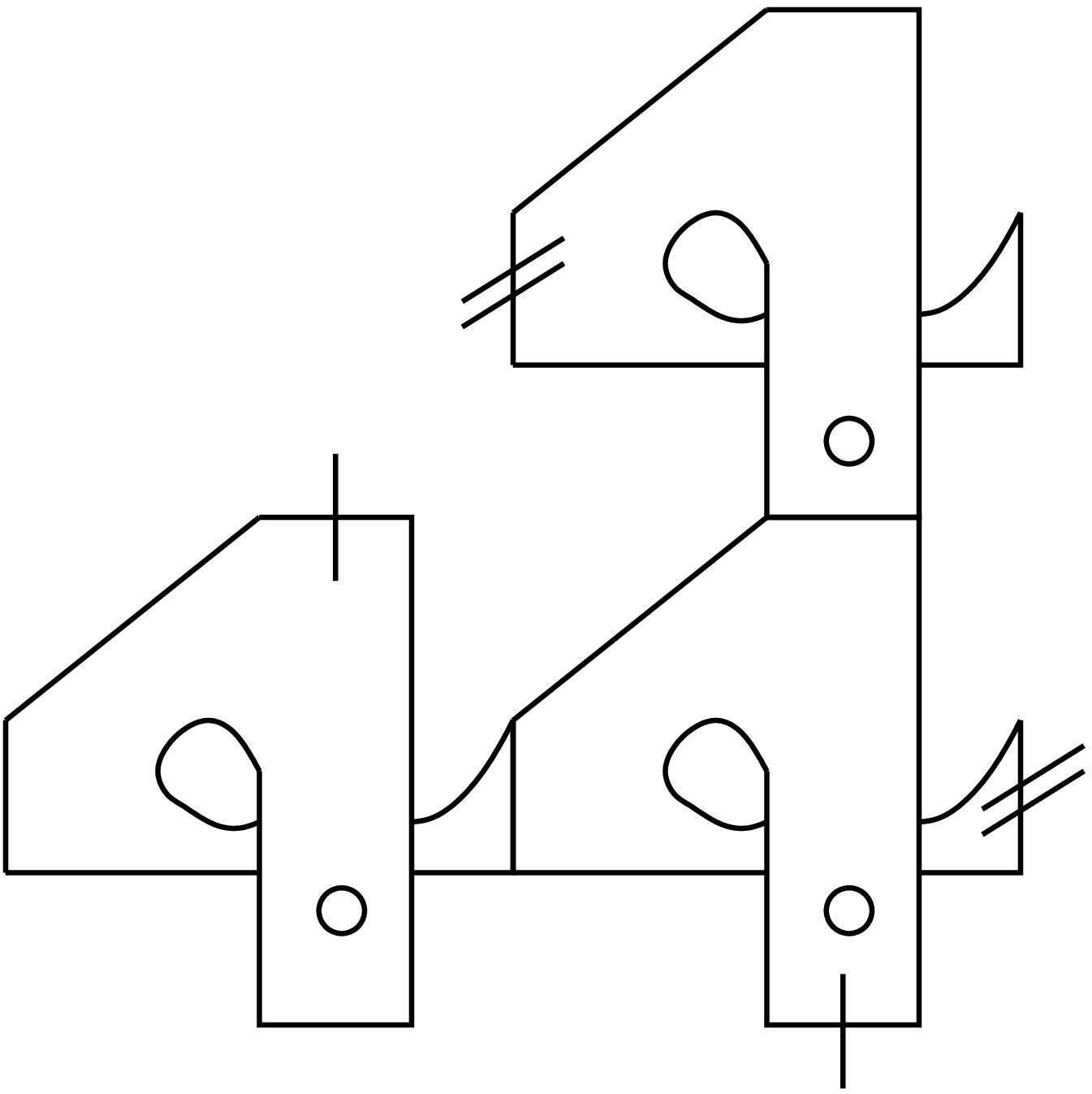}\hspace{15mm}\includegraphics[scale=0.25]{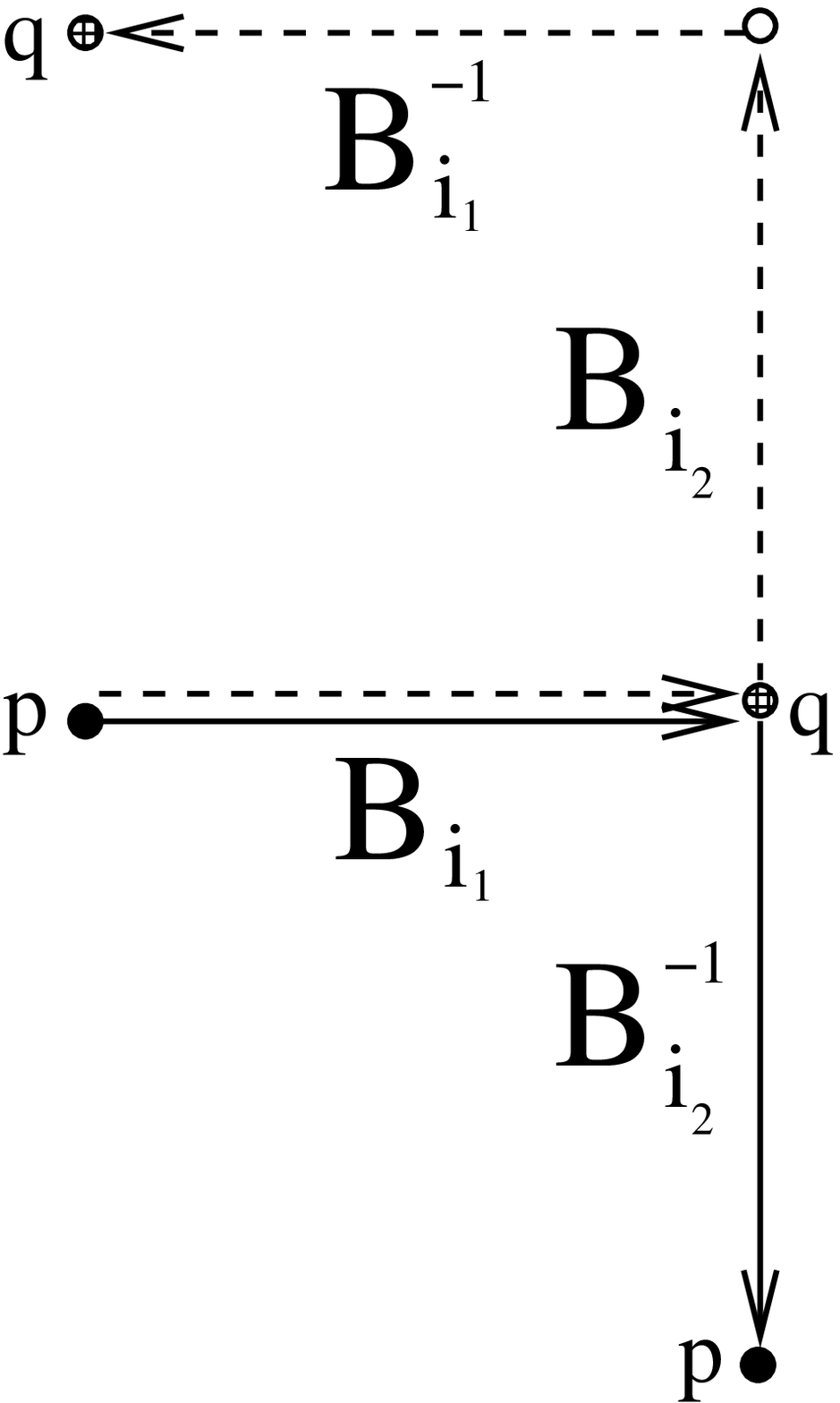}\\
Figure 16: The domain   $\rm U(B_{i_1}^{-1}B_{i_2}B_{i_1}B_{i_2}^{-1}B_{i_1})$.
\end{center}
\indent which is homeomorphic to a "punctured" torus as well.\\
\textbf{Subcase (2.4)}. In this case, $\rm P$ has at least four sides. In order to "see" the punctured torus we present an argumentation based on the pentagon presented in figure 5. However, it remains valid for any polygon having more than three sides. If $\rm\lambda_{i_1}+\lambda_{i_2}$ is an integer, the lift of  $\rm B_{i_2}B_{i_1}$ to the covering $\rm X(P_0)$ via $\hat{\pi}$ is a loop. The domain $\rm U(B_{i_2}B_{i_1})$ is then obtained from two copies of figure 5 after the following identifications:
\begin{center}
\hspace{1mm}\includegraphics[scale=0.2]{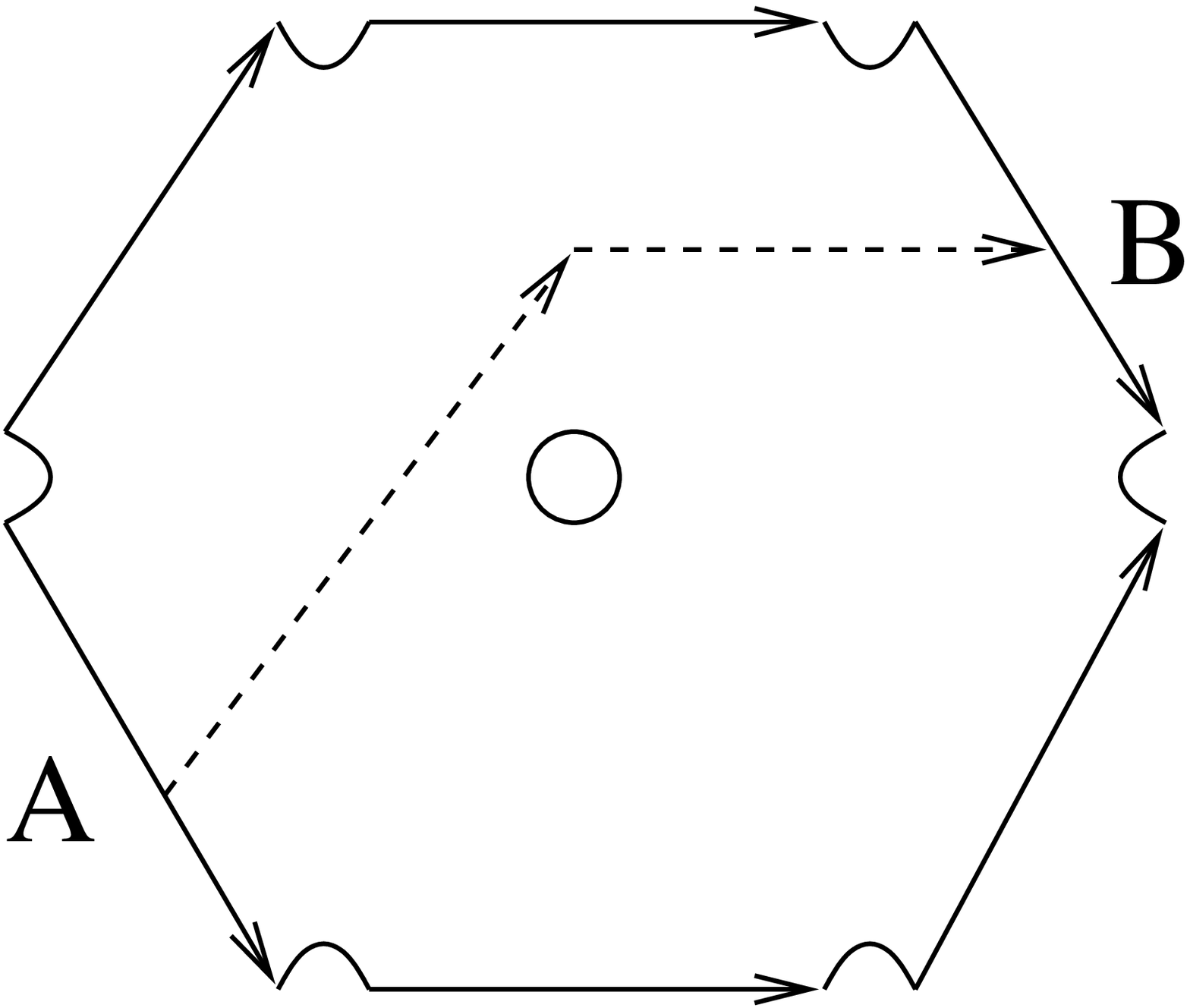}\hspace{15mm}\hspace{1mm}\includegraphics[scale=0.2]{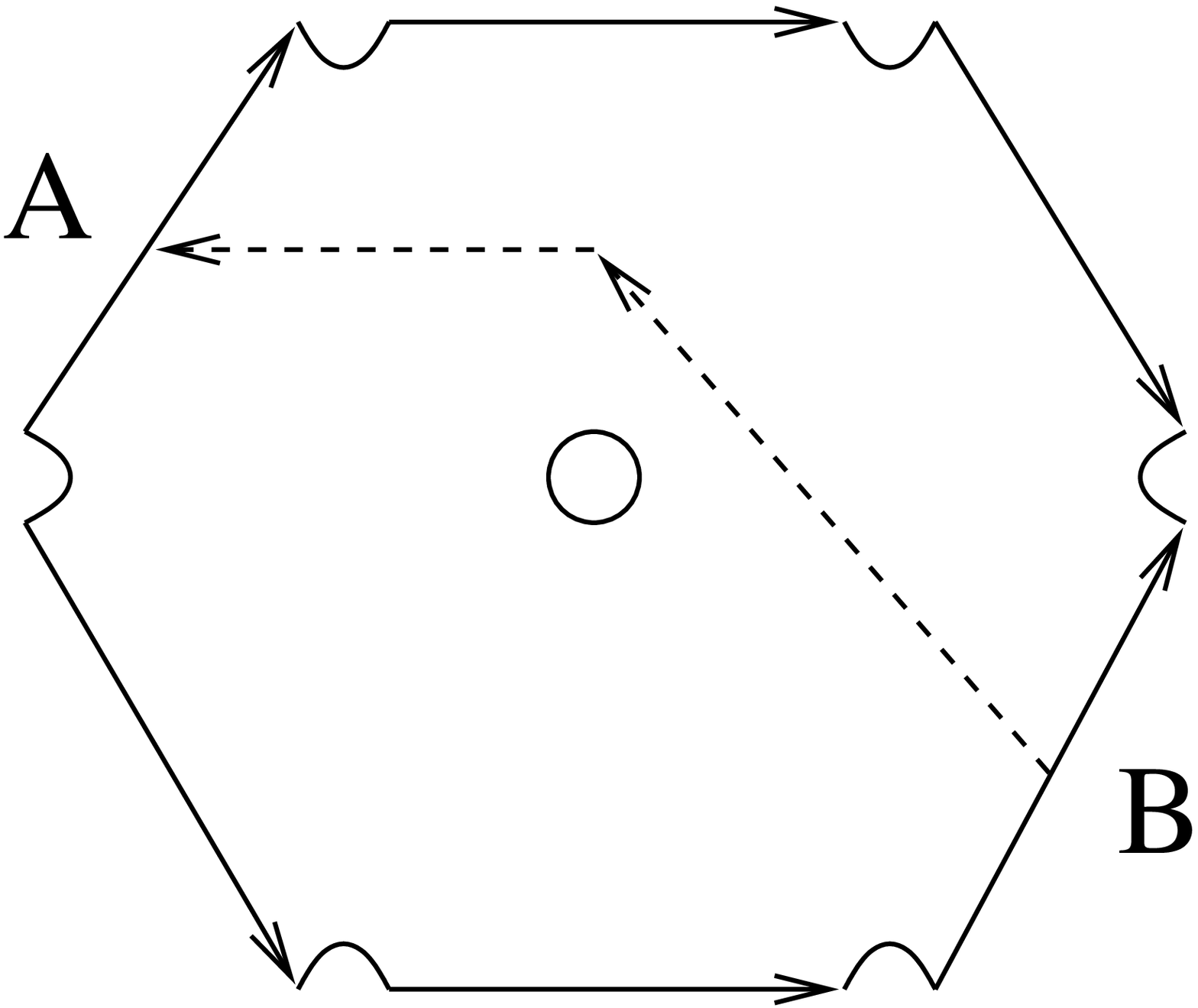}\\
Figure 17: the domain $\rm U(B_{i_2}B_{i_1})$.
\end{center}
\indent Clearly, $\rm U(B_{i_2}B_{i_1})$ is homeomorphic to a disc from which we have removed some discs. In the following figure we show the identifications defining the domain  $\rm U(B_{i_1}B_{i_2}^2B_{i_1})$
\begin{center}
\hspace{1mm}\includegraphics[scale=0.2]{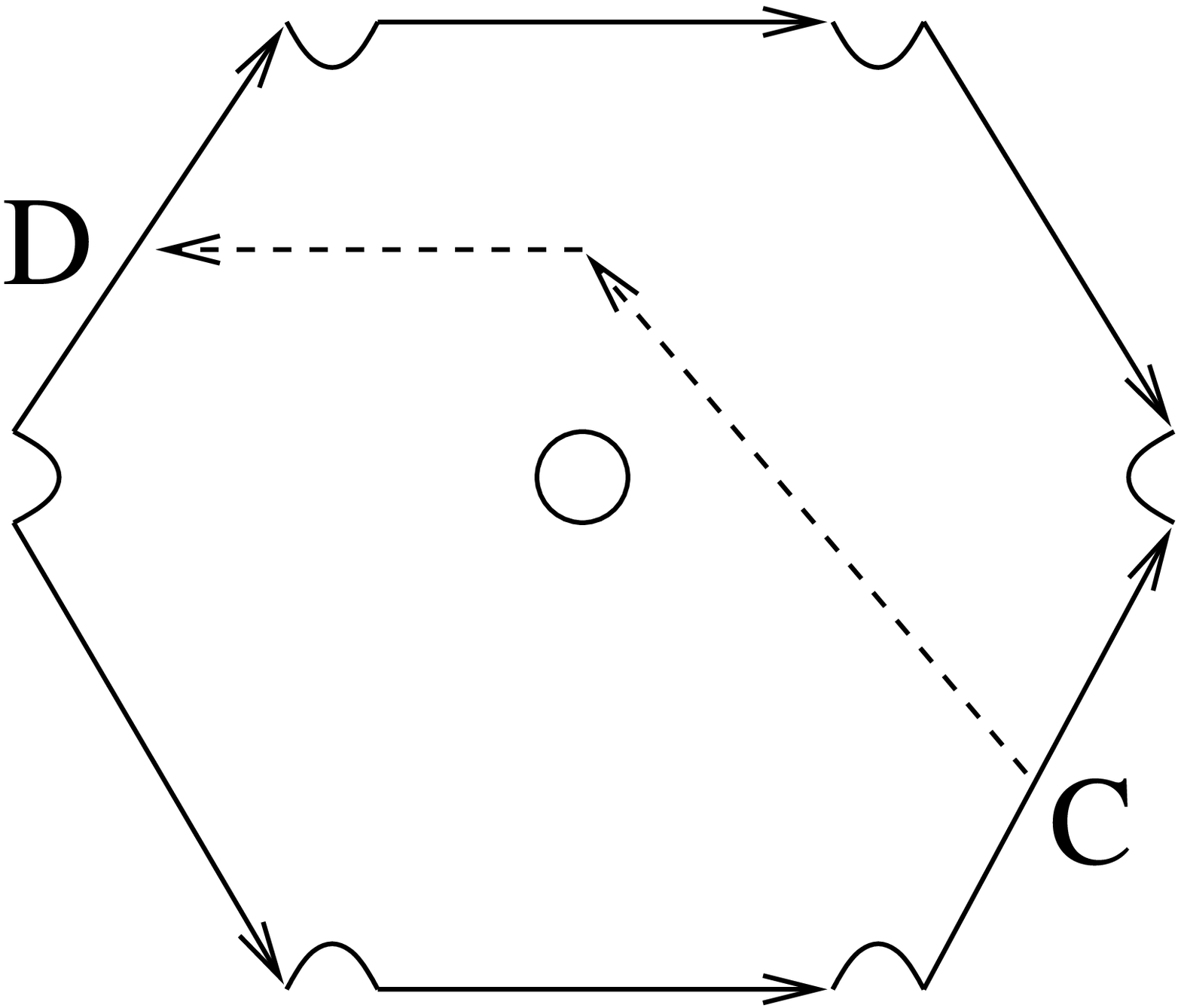}\hspace{5mm}\includegraphics[scale=0.2]{p2}\hspace{5mm}\includegraphics[scale=0.2]{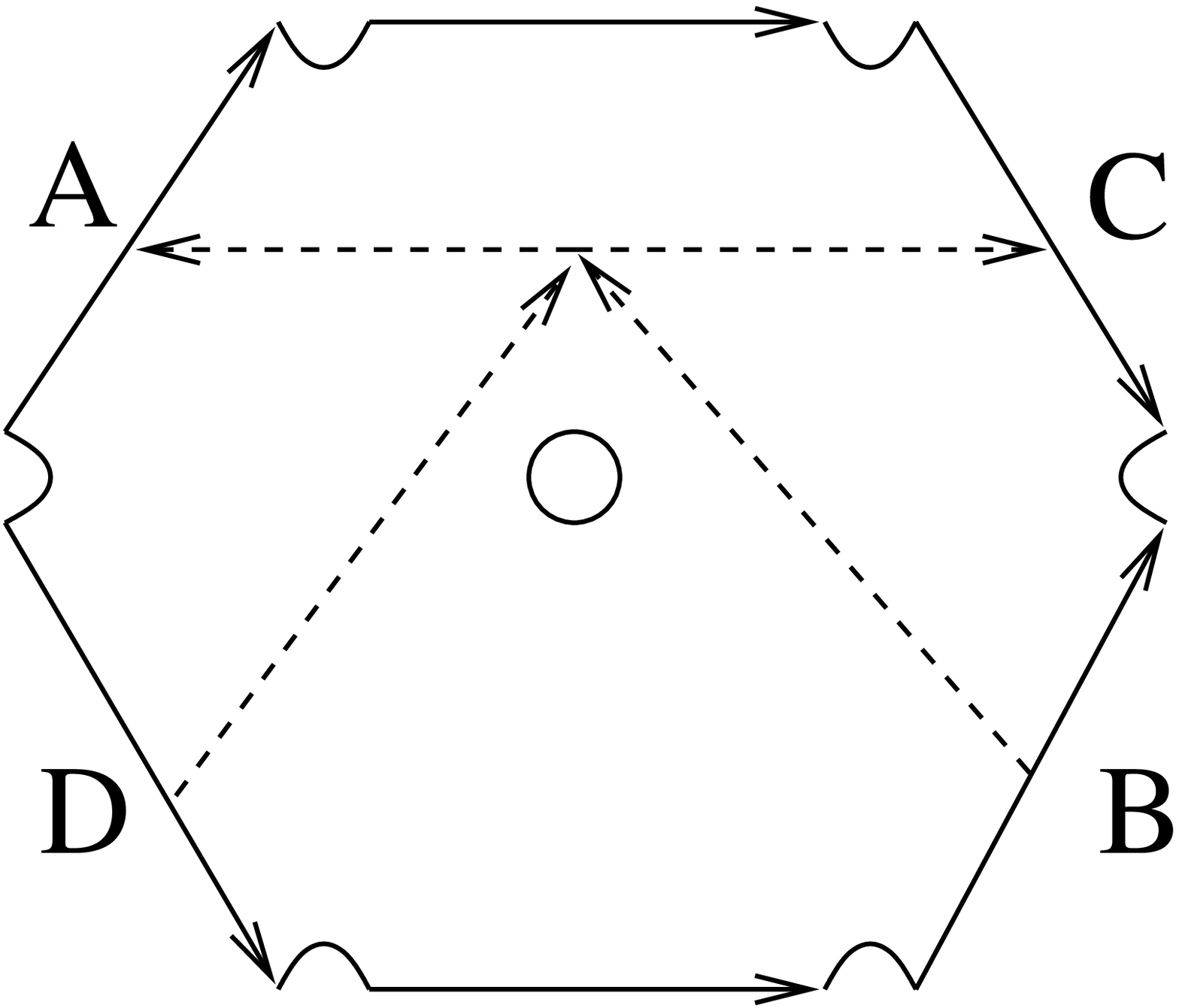}\\
Figure 18: the domain $\rm U(B_{i_1}B_{i_2}^2B_{i_1})$.
\end{center}
which is homeomorphic to torus to which we have removed a finite number of discs.\\
\qed

\end{document}